\documentclass[opre,nonblindrev]{informs3} 

\OneAndAHalfSpacedXI 

\usepackage{amssymb, amsmath, amsfonts}
\usepackage{mathrsfs, color, bm, enumitem, url, xr}
\usepackage{graphicx, transparent} 
\usepackage{multirow, titlecaps, accents, epstopdf}
\usepackage{caption}
\usepackage{subcaption}
\usepackage{alphalph,etoolbox}
\patchcmd{\subequations}{\alph{equation}}{\alphalph{\value{equation}}}{}{}

\graphicspath{{figures/}}

\usepackage{algorithmic}            
\usepackage{algorithm}

\usepackage[utf8]{inputenc}    
\usepackage{hyperref}          
\usepackage{url}               
\usepackage{booktabs}          
\usepackage{amsfonts}          
\usepackage{nicefrac}          
\usepackage{microtype}         



\makeatletter
\def\munderbar#1{\underline{\sbox\tw@{$#1$}\dp\tw@\z@\box\tw@}}
\makeatother

\newcommand{\bea}{\begin{equation*}\begin{aligned}}
\newcommand{\eea}{\end{aligned}\end{equation*}}

\usepackage{url}
\usepackage{natbib}
\usepackage{comment} 
\usepackage{graphicx}
\usepackage{appendix}

\bibpunct[, ]{(}{)}{,}{a}{,}{,}%
\def\bibfont{\small}%
\setcitestyle{citesep={;}}

\TheoremsNumberedThrough     
\ECRepeatTheorems

\EquationsNumberedThrough    

\usepackage{xcolor}

\usepackage[normalem]{ulem}

\hypersetup{
	colorlinks=true,
	citecolor=blue,
	urlcolor=blue,
	linkcolor=blue,
	pdfborder={0 0 0},
}

\begin{document}
	
	
	\RUNTITLE{A novel exact approach to polynomial optimization}
	\TITLE{
		A novel exact approach to polynomial optimization
	}
	
	\ARTICLEAUTHORS{
 		\AUTHOR{Dimitris Bertsimas}
		\AFF{Operations Research Center, Massachusetts Institute of Technology, United States,
  \EMAIL{d.bertsim@mit.edu}}
		\AUTHOR{Dick den Hertog} 
		\AFF{Amsterdam Business School, University of Amsterdam, The Netherlands,
			\EMAIL{d.denhertog@uva.nl}}
    		\AUTHOR{Thodoris Koukouvinos}
		\AFF{Operations Research Center, Massachusetts Institute of Technology, United States,
  \EMAIL{tkoukouv@mit.edu}}
 	}
	\ABSTRACT{Polynomial optimization problems represent a wide class of optimization problems, with a large number of real-world applications. Current approaches for polynomial optimization, such as the sum of squares (SOS) method, rely on large-scale semidefinite programs, and therefore the scale of problems to which they can be applied is limited and an optimality guarantee is not always provided. {\color{black} Moreover, the problem can have other convex nonlinear parts, that cannot be handled by these approaches.} In this paper, we propose an alternative approach for polynomial optimization. We obtain a convex relaxation of the original polynomial optimization problem, by deriving a sum of linear times convex (SLC) functions decomposition for the polynomial. We prove that such SLC decompositions always exist for arbitrary degree polynomials. {\color{black} Moreover, we derive the SLC decomposition that results in the tightest lower bound, thus improving significantly the quality of the obtained bounds in each node of the spatial Branch and Bound method.} In the numerical experiments, we show that our approach outperforms state-of-the-art methods for polynomial optimization, such as BARON and SOS. We show that with our method, we can solve polynomial optimization problems {\color{black} to optimality} with 40 variables and degree 3, as well as 20 variables and degree 4, in less than an hour.}
		
	%
	
	
	
	
	%
	
	\maketitle

\section{Introduction}  
In this paper, we consider the following optimization problem 
\begin{equation}
    \begin{array}{cll}
        \displaystyle \min_{\boldsymbol{x}} &\quad p_d(\boldsymbol{x}) \\
        {\rm s.t.} 	&\quad  \boldsymbol{f} \left( \boldsymbol{x} \right) \le \boldsymbol{0}, \\ 
        &\quad \boldsymbol{x} \in \mathcal{X},
    \end{array} \label{eq:genericproblem}
\end{equation} 
where $p_d(\boldsymbol{x})$ represents a multivariate polynomial of degree $d$. {\color{black} We have $\boldsymbol{f} \left( \boldsymbol{x} \right) = [f_1(\boldsymbol{x}) \ \cdots \ f_K (\boldsymbol{x})]^T$, and $f_k: \mathbb{R}^{n} \rightarrow (-\infty, + \infty ]$ are polynomials, for every $k \in \{1,\ldots,K\}.$. Moreover, the set $\mathcal{X} \subseteq \mathbb{R}^{n_x}$ is defined by:
\begin{align*}
\mathcal{X} = \{ \boldsymbol{x} \in \mathbb{R}^{n_x} \mid \boldsymbol{A} \boldsymbol{x} \le \boldsymbol{b}, \: \: \boldsymbol{h} (\boldsymbol{x}) \le \boldsymbol{0} \},    
\end{align*}
where $\boldsymbol{A} \in \mathbb{R}^{k \times n_x}$, $\boldsymbol{b} \in \mathbb{R}^{k}$, $\boldsymbol{h} (\boldsymbol{x}) = [h_1(\boldsymbol{x})  \ \cdots \ h_J (\boldsymbol{x})]^T$, and $h_j: \mathbb{R}^{n_x} \rightarrow (-\infty, + \infty ]$ are proper, closed and convex for every $j \in \mathcal{J} = \{1,\ldots,J\}$.}

Problem \eqref{eq:genericproblem} is in general nonconvex and has many real-world applications, including high-order moment portfolio optimization problems \cite{niu2019high, pham2011efficient}, geometric programs \cite{boyd2007tutorial}, optimal power flow in electrical networks \cite{lavaei2011zero}, robot motion planning \cite{majumdar2014control}, and matrix completion problems \cite{bakonyi1995euclidian}.

Several methods have been proposed to solve polynomial optimization problems, those are specific versions of Problem \eqref{eq:genericproblem}, where $h_j(\boldsymbol{x}) = 0$, for all $j \in \mathcal{J}$. One state-of-the-art approach is the sum-of-squares (SOS) algorithm \cite{parrilo2003semidefinite}. The SOS method obtains a relaxation of the original polynomial optimization problem, by constructing a semidefinite program (SDP). Moreover, in order to converge to the global optimal solution, a hierarchy of relaxations, of increasing size, needs to be considered. As a result, SOS becomes intractable when the size of the polynomial optimization problem increases. Moreover, the SOS method may or may not provide an optimality guarantee, while it is designed to handle only linear or polynomial constraints. On the other hand, our method can handle general convex constraints and does give an optimality guarantee. Another approach is the Lasserre’s hierarchy method \cite{lasserre2001global, lasserre2008semidefinite, lasserre2009moments}, that constructs a hierarchy of SDP relaxations that converge to the global optimal solution. The main drawback of Lasserre’s hierarchy approach, is that it requires a high order SDP relaxation in order to obtain a good approximation and therefore it quickly becomes computationally intractable. The Reformulation-Linearization-Technique (RLT) has also been proposed for polynomial optimization problems \cite{sherali2013}. {\color{black} We note that both Lassere's hierarchy and RLT, can be viewed as special cases of the SOS method, and therefore they are outperformed by the latter, see \cite{parrilo2003semidefinite}}. Finally, we note that state-of-the-art global optimization solvers, such as BARON \cite{sahinidis1996baron}, are also good alternatives for polynomial optimization problems.

Recently, \cite{bertsimas2023novel} introduced a novel method to solve nonconvex optimization problems, in which the nonconvex functions can be written as the sum of linear times convex (SLC) functions. They introduced the Reformulation- Perspectification-Technique (RPT) in order to obtain a convex relaxation of the original nonconvex optimization problem, and then used branch and bound to obtain the global optimal solution. The main idea of the current paper is to derive SLC decompositions for polynomials, and then leverage the RPT approach in order to obtain a convex relaxation of Problem \eqref{eq:genericproblem}.

In this paper, we focus on deriving SLC decompositions for polynomials. We propose two types of SLC decompositions and prove the existence for both of them, for arbitrary degree polynomials. {\color{black} Moreover, we outline how to derive the "best" SLC decomposition out of the infinitely many possible ones, that is the decomposition that results in the tightest lower bound for Problem \eqref{eq:genericproblem}.} As we show in the numerical experiments, the "best" SLC decomposition provides very high quality bounds and as a result we can easily find the global optimal solution of Problem \eqref{eq:genericproblem}, without needing to do branch and bound in many cases.

\medskip\noindent 
{\bf Contributions.}
Our main contributions can be summarized as follows: 
\begin{enumerate} 
\item We prove the existence of two types of SLC decompositions for polynomials of arbitrary degree. In the first one, the linear terms are defined by the linear constraints of the feasible region, and the convex terms are polynomials of one degree less than the original one. {\color{black} Further, in the second one, the linear terms are defined as linearized products of $d-2$ linear constraints that define $\mathcal{X}$, and the convex terms are quadratics.} 
\item We derive the SLC decomposition that results in the tightest lower bound for Problem \eqref{eq:genericproblem}, out of the infinitely many possible decompositions. The resulting problem is an SDP, in which the size of the largest LMI is $\mathcal{O}(n^2)$, that is an improvement over SOS in which it is at least $\mathcal{O}(n^d)$, where $d$ is the degree of the polynomial. 
\item We conduct numerical experiments on polynomial optimization problems and demonstrate that our approach often outperforms BARON and SOS. We show that with our approach, we can solve polynomial optimization problems with 40 variables in the case of degree 3, and 20 variables in the case of degree 4, both in less than an hour. 
\end{enumerate}   

\vspace{0.3cm}
\noindent  {\color{black} This paper is structured as follows: In Section 2, we outline our method for polynomials of degree 3. In Section 3, we describe our method for polynomials of degree 4. In Section 4, we generalize our method to arbitrary degree polynomials. In Section 5, we provide insights on how our approach compares to current methods from the literature. In Section 6, we asses the numerical performance of the approach, and finally in Section 7 we summarize our findings.}

\medskip\noindent
{\bf Notation.} 
We use bold faced characters such as $ \boldsymbol{x} $ to represent vectors and bold faced capital letters such as $ \boldsymbol{X} $ to represent matrices. The calligraphic letters~$\mathcal{I}$, $\mathcal{J}$, $\mathcal{K}$, $\mathcal{L}$ and the corresponding capital Roman letters $I$, $J$, $K$, $L$ are reserved for finite index sets and their respective cardinalities, i.e., $\mathcal{I}= \{1,\dots, I\}$ etc. Moreover, the caligraphic letters $\mathcal{X}, \mathcal{Z},\mathcal{V}$ are used to denote feasible regions. We use the notation $p_d(\boldsymbol{x})$ for a multivariate polynomial of degree $d$.

\section{Polynomials of degree 3}  
We consider a generic polynomial of degree 3, that is,  
\begin{align*}
    p_3(\boldsymbol{x}) = \sum_{i \le j \le k} c_{ijk}^3 x_i x_j x_k + \boldsymbol{x}^T \boldsymbol{c}^2 \boldsymbol{x} + \boldsymbol{x}^T \boldsymbol{c}^1 + c^0,
\end{align*}    
where $\boldsymbol{c}^3 \in \mathbb{R}^{n \times n \times n}, \boldsymbol{c}^2 \in \mathbb{R}^{n \times n}, \boldsymbol{c}^1 \in \mathbb{R}^n, c^0 \in \mathbb{R}$. We next discuss how to derive valid SLC decompositions for $p_3$, where the linear terms are defined by linear constraints of the feasible region, and the convex terms are polynomials of degree 2. {\color{black} In the following, we assume that the feasible region is contained in the interval $[0,1]^n$. Boundedness of the feasible region is already sufficient to calculate a box $[l,u]^n$ that contains the feasible region, and then this can simply be transformed to the unit box $[0,1]^n$, see \cite{burer2009nonconvex}.} 

\subsection{Existence of SLC decomposition}
We assume that $\mathcal{X} \subseteq [0,1]^n$. In this case, we will show that we can write $p_3(\boldsymbol{x})$ as the sum of linear functions, defined by $x_i, 1-x_i, 1, \: i \in [n]$, multiplied by convex polynomials of degree 2, that is,
\begin{equation}
    p_3(\boldsymbol{x}) = \sum_{i=1}^n x_ip_2^i(\boldsymbol{x}) + \sum_{i=1}^n (1-x_i)q_2^i(\boldsymbol{x}) + \beta_2(\boldsymbol{x}), \label{eq:poly3_decomp_gen}
\end{equation}
where  
\begin{align*}
    & p_2^i(\boldsymbol{x}) = \boldsymbol{x}^T \boldsymbol{P}^i \boldsymbol{x} + \boldsymbol{x}^T \boldsymbol{r}^i + w^i, \\ 
    & q_2^i(\boldsymbol{x}) = \boldsymbol{x}^T \boldsymbol{Q}^i \boldsymbol{x} + \boldsymbol{x}^T \boldsymbol{f}^i + g^i, \\ 
    & \beta_2(\boldsymbol{x}) = \boldsymbol{x}^T \boldsymbol{P}' \boldsymbol{x} + \boldsymbol{x}^T \boldsymbol{r}' + w',
\end{align*}  
and all polynomials $p_2^i(\boldsymbol{x}), q_2^i(\boldsymbol{x}), \beta_2(\boldsymbol{x})$ are convex. We have the following result.
\begin{theorem}
\label{thm:exist_degree3}
Every polynomial of degree 3, denoted as $p_3(\boldsymbol{x})$, can be written as 
   \begin{align*}
    p_3(\boldsymbol{x}) = \sum_{i=1}^n x_ip_2^i(\boldsymbol{x}) + \sum_{i=1}^n (1-x_i)q_2^i(\boldsymbol{x}) + 1 \beta_2(\boldsymbol{x}),
\end{align*}  
where $p_2^i(\boldsymbol{x}),q_2^i(\boldsymbol{x}),\beta_2(\boldsymbol{x})$ are convex polynomials of degree 2, for all $i \in [n]$. 
\end{theorem}  

\begin{proof}{Proof. }
{\color{black} We define the matrices $\boldsymbol{F}^i$ such that  $F_{jk}^i = c_{ijk}^3$, for all $i \le j \le k \in [n]$. Moreover, we take $\boldsymbol{B} = \boldsymbol{c}^2, \: \boldsymbol{b} = \boldsymbol{c}^1, \: \eta = c^0$. We consider the degree 2 polynomials 
\begin{align*}
    & p_2^i(\boldsymbol{x}) = \boldsymbol{x}^T \boldsymbol{F}^i \boldsymbol{x} + \alpha \|\boldsymbol{x}\|^2 - \alpha (n+1) x_i, \\ 
    & q_2^i(\boldsymbol{x}) = \alpha \|\boldsymbol{x}\|^2, \\ 
    & \beta_2(\boldsymbol{x}) = \boldsymbol{x}^T \boldsymbol{B} \boldsymbol{x} + \boldsymbol{x}^T \boldsymbol{b} + \eta + \alpha \|\boldsymbol{x}\|^2,
\end{align*}   
where $\alpha_p^i = \max_k \left\{ \sum_{j \neq k} |F_{kj}^i| - F_{kk}^i \right\}, \: \alpha_{\beta} = \max_k \left\{ \sum_{j \neq k} |B_{kj}| - B_{kk} \right\}$ and $\alpha = \max \left(0, \: \max_i \alpha_p^i, \: \alpha_{\beta} \right)$. By definition, the Hessian matrices of the quadratics are diagonally dominant, and therefore from the Gershgorin Theorem \cite{gershgorin1931uber} it follows that they are positive semi-definite. Therefore, the polynomials $p_2^i, q_2^i, \beta_2$ are convex.} Further, they satisfy
\begin{align*}
    \sum_{i=1}^n x_ip_2^i(\boldsymbol{x}) + \sum_{i=1}^n(1-x_i)q_2^i(\boldsymbol{x}) + \beta_2(\boldsymbol{x})  
    \: \: = \: \: p_3(\boldsymbol{x}).
\end{align*}  \hfill $\square$
\end{proof}     

\begin{remark}
We note that the result of Theorem \ref{thm:exist_degree3} also holds, if we only consider the terms in the objective, defined by $x_i, \: (1-x_i),\: 1, \: i \in [n]$. However, by considering the additional term $1 \beta_2(\boldsymbol{x})$, we obtain a stronger approximation.
\end{remark}

So far we have shown that an SLC decomposition for $p_3(\boldsymbol{x})$ always exists. We next provide an illustrative example on how to construct one, for the case of a polynomial of degree 3 in $\mathbb{R}^2$. 

\begin{example}
Consider the generic polynomial 
\begin{align*}
    p_3(\boldsymbol{x}) = a_1 x_1^3 + a_2 x_2^3 + a_3 x_1^2 x_2 + a_4 x_1 x_2^2 + a_5 x_1^2 + a_6 x_2^2 + a_7 x_1 x_2 + a_8 x_1 + a_9 x_2,
\end{align*}  
which covers all possible polynomials of degree 3, in $\mathbb{R}^2$. We initialize $\boldsymbol{P}^1 = \boldsymbol{r}^1 = w^1 = \boldsymbol{P}^2 = \boldsymbol{r}^2 = w^2 = \boldsymbol{Q}^1 = \boldsymbol{f}^1 = g^1 = \boldsymbol{Q}^2 = \boldsymbol{f}^2 = g^2 = \boldsymbol{P}' = \boldsymbol{r}' = w' = \boldsymbol{0}$.
We take 
\begin{align*}
    & P_{11}^1 = a_1, \: P_{12}^1 = P_{21}^1 = a_3/2, \: P_{22}^1 = a_4, \: P_{22}^2 = a_2, \\ 
    & r_{1}^1 = a_5, \: r_{2}^1 = a_7, \: r_{2}^2 = a_6, \: w^1 = a_8, \: w^2 = a_9.
\end{align*}   
As a result, we obtain
\begin{align*}
  & p_2^1(\boldsymbol{x}) = a_1 x_1^2 + a_3 x_1 x_2 + a_4 x_2^2 + a_5 x_1 + a_7 x_2 + a_8, \\ 
  & p_2^2(\boldsymbol{x}) = a_2 x_2^2 + a_6 x_2 + a_9, \\ 
  & x_1 p_2^1(\boldsymbol{x}) + x_2 p_2^2(\boldsymbol{x}) = p_3(\boldsymbol{x}).  
\end{align*}  
Then, we can make the polynomials of degree 2 convex, as illustrated in the proof of Theorem \ref{thm:exist_degree3}. \hfill $\square$ 
\end{example} 

Observe that the SLC decomposition for a polynomial of degree 3 is not unique and in fact there are infinitely many choices for the coefficients of the convex quadratics that define valid SLC decompositions. The question that arises in this case is how we can obtain the "best" SLC decomposition out of the infinitely many possible ones, that is, the decomposition that results in the tightest lower bound for Problem \eqref{eq:genericproblem}.

\subsection{Deriving the best SLC decomposition}
Observe that all valid SLC decompositions of the proposed form are parametrized by coefficients, denoted as 
$\boldsymbol{Z}$, that satisfy certain equality constraints, those of matching the coefficients of $p_3(\boldsymbol{x})$. Further, in order to ensure convexity for the polynomials of degree 2, we impose the LMIs: $ \boldsymbol{P}^i \succeq \boldsymbol{0}, \ \boldsymbol{Q}^i \succeq \boldsymbol{0}, \boldsymbol{P}' \succeq \boldsymbol{0}, \: \: i \in [n].$  Let $\mathcal{Z}$ denote the set of possible realizations for $\boldsymbol{Z}$, that can be described as follows
\begin{equation*}
    \mathcal{Z} = \left\{ \begin{array}{c} 
    \boldsymbol{P}^i, \boldsymbol{Q}^i, \boldsymbol{P}' \in \mathbb{R}^{n \times n}, \\
    \boldsymbol{r}^i, \boldsymbol{f}^i, \boldsymbol{r}' \in \mathbb{R}^n, \\
   w^i, g^i,  w' \in \mathbb{R}, \end{array} \enskip : \enskip  \begin{array}{lr}
	 \sum_{i=1}^n \left \langle \boldsymbol{A}_j^i, \boldsymbol{P}^i \right \rangle + \left \langle \boldsymbol{B}_j^i, \boldsymbol{Q}^i \right \rangle + \left( \boldsymbol{c}_j^i \right)^T \boldsymbol{r}^i \\ 
     \quad \quad + \left( \boldsymbol{d}_j^i \right)^T \boldsymbol{f}^i + \mu_j^i w^i + \nu_j^i g^i \\ 
     + \left \langle \boldsymbol{\Xi}_j, \boldsymbol{P}' \right \rangle + \boldsymbol{\omega}_j^T \boldsymbol{r}' + \gamma_j w' = s_j, \quad j \in [m], \\
    \boldsymbol{P}^i, \boldsymbol{Q}^i, \boldsymbol{P}' \succeq \boldsymbol{0}, \ i \in [n].
	 \end{array}  \right\}
\end{equation*} 

{\color{black} In the definition of the set $\mathcal{Z}$, we use the notation $\boldsymbol{A}_j^i,\boldsymbol{B}_j^i,\boldsymbol{\Xi}_j,\boldsymbol{c}_j^i,\boldsymbol{d}_j^i,\boldsymbol{\omega}_j,\mu_j^i,\nu_j^i,\gamma_j$ for the coefficients that multiply $\boldsymbol{P}^i,\boldsymbol{Q}^i,\boldsymbol{P}',\boldsymbol{r}^i,\boldsymbol{f}^i,\boldsymbol{r}',w^i,g^i,w'$, respectively, in the $j$-th constraint, while $s_j$ denotes the right hand side.}

{\color{black} After obtaining an SLC decomposition, we leverage RPT to obtain a convex relaxation of the original problem. We multiply and divide the arguments of the quadratics with the linear terms multiplying them and linearize $\boldsymbol{U} = \boldsymbol{x} \boldsymbol{x}^T$, to obtain  
\begin{align*}
   p_3(\boldsymbol{x}) = \sum_{i=1}^n x_i p_2^i \left( \frac{\boldsymbol{U}_i}{x_i} \right) + \sum_{i=1}^n (1-x_i) q_2^i \left( \frac{\boldsymbol{x}-\boldsymbol{U}_i}{1-x_i} \right) + \beta_2(\boldsymbol{x}). 
\end{align*}  
We then remove the nonconvex constraint $\boldsymbol{U} = \boldsymbol{x} \boldsymbol{x}^T$, and try to achieve it by multiplying constraints in the feasible region and linearizing products with new variables, see \cite{bertsimascone}. Let $\overline{\mathcal{X}}$ denote the feasible region resulting from the constraint multiplications. In this case, for any $\boldsymbol{Z} \in \mathcal{Z}$, we obtain the following convex relaxation 
\begin{align*}
    g_3(\boldsymbol{x},\boldsymbol{U},\boldsymbol{Z}) &= \sum_{i=1}^n \frac{1}{x_i} \boldsymbol{U}_i^T \boldsymbol{P}^i \boldsymbol{U}_i + \boldsymbol{U}_i^T \boldsymbol{r}^i + w^i x_i + \frac{1}{1-x_i} \left( \boldsymbol{x} - \boldsymbol{U}_i \right)^T \boldsymbol{Q}^i \left( \boldsymbol{x} - \boldsymbol{U}_i \right) \\
    &\quad \quad + \left( \boldsymbol{x} - \boldsymbol{U}_i \right)^T \boldsymbol{f}^i + (1-x_i)g^i + \boldsymbol{x}^T \boldsymbol{P}' \boldsymbol{x} + \boldsymbol{x}^T \boldsymbol{r}' + w'.
\end{align*}    
Since any $\boldsymbol{Z} \in \mathcal{Z}$ gives us a convex relaxation, the questions that arises in this case, is which $\boldsymbol{Z}$ should we select. To answer this question, we view the problem as a robust optimization problem, in which $\boldsymbol{Z}$ are the uncertain parameters, and our goal is to optimize for all realizations $\boldsymbol{Z} \in \mathcal{Z}$, that is,
\begin{equation}
       \min_{(\boldsymbol{x},\boldsymbol{U}) \in \overline{\mathcal{X}}} \: \max_{\boldsymbol{Z} \in \mathcal{Z}} \:g_3(\boldsymbol{x},\boldsymbol{U},\boldsymbol{Z}). \label{eq:minmaxproblem_degree3}
\end{equation} 
Observe that in the above problem we find the tightest lower bound, out of the infinitely many possible ones. We next show that the above problem can be reformulated as a semi-definite optimization problem using duality.} We have the following result.
\begin{theorem}
\label{thm:reformulation_degree3}
{\color{black} $\left(\boldsymbol{x},\boldsymbol{U}\right)$ is an optimal solution for Problem \eqref{eq:minmaxproblem_degree3} if and only if $\exists \left( \boldsymbol{Y}, \boldsymbol{R}, \boldsymbol{\lambda} \right)$ such that $\left(\boldsymbol{x},\boldsymbol{U},\boldsymbol{Y},\boldsymbol{R},\boldsymbol{\lambda}\right)$ is an optimal solution for the following problem:}   
\begin{subequations}
			\begin{align}
				\displaystyle \min_{\boldsymbol{x}, \boldsymbol{U}, \boldsymbol{Y}, \boldsymbol{R}, \boldsymbol{\lambda}} &\quad - \sum_{j=1}^{m} \lambda_j s_j \nonumber \\
				{\rm s.t.} &\quad \boldsymbol{Y}^i + \sum_{j=1}^{m} \lambda_j \boldsymbol{A}_j^i \preceq \boldsymbol{0}, \quad i \in [n], \label{eq:best_slc_degree3_a} \\ 
                &\quad \boldsymbol{R}^i + \sum_{j=1}^{m} \lambda_j \boldsymbol{B}_j^i \preceq \boldsymbol{0}, \quad i \in [n],  \label{eq:best_slc_degree3_b} \\ 
                &\quad \boldsymbol{Y}^{n+1} + \sum_{j=1}^{m} \lambda_j \boldsymbol{\Xi}_j \preceq \boldsymbol{0}, \label{eq:best_slc_degree3_c} \\ 
                 &\quad \boldsymbol{U}_i + \sum_{j=1}^{m} \lambda_j \boldsymbol{c}_j^i = \boldsymbol{0}, \quad i \in [n], \label{eq:best_slc_degree3_d} \\ 
                &\quad \boldsymbol{x}-\boldsymbol{U}_i + \sum_{j=1}^{m} \lambda_j \boldsymbol{d}_j^i = \boldsymbol{0}, \quad i \in [n], \label{eq:best_slc_degree3_e} \\  
                &\quad \boldsymbol{x} + \sum_{j=1}^{m} \lambda_j \boldsymbol{\omega}_j = \boldsymbol{0}, \label{eq:best_slc_degree3_f} \\ 
                &\quad x_i + \sum_{j=1}^{m} \lambda_j \mu_j^i = 0, \quad i \in [n], \label{eq:best_slc_degree3_g} \\ 
                &\quad 1-x_i + \sum_{j=1}^{m} \lambda_j \nu_j^i = 0, \quad i \in [n], \label{eq:best_slc_degree3_h} \\ 
                &\quad 1 + \sum_{j=1}^{m} \lambda_j \gamma_j = 0, \label{eq:best_slc_degree3_i} \\ 
                &\quad \begin{pmatrix}
                        \boldsymbol{Y}^i & \boldsymbol{U}_i \\
                        \boldsymbol{U}_i^T & x_i
                   \end{pmatrix} \succeq \boldsymbol{0}, \quad i \in [n], \label{eq:best_slc_degree3_j} \\ 
               &\quad \begin{pmatrix}
                        \boldsymbol{R}^i & \boldsymbol{x}-\boldsymbol{U}_i \\
                        \left(\boldsymbol{x}-\boldsymbol{U}_i\right)^T & 1-x_i
                   \end{pmatrix} \succeq \boldsymbol{0}, \quad i \in [n], \label{eq:best_slc_degree3_k} \\ 
             &\quad \begin{pmatrix}
                        \boldsymbol{Y}^{n+1} & \boldsymbol{x} \\
                        \boldsymbol{x}^T & 1
                   \end{pmatrix} \succeq \boldsymbol{0}, \label{eq:best_slc_degree3_l} \\
              &\quad \left( \boldsymbol{x}, \boldsymbol{U} \right) \in \overline{\mathcal{X}}. \label{eq:best_slc_degree3_m}
			\end{align}
\end{subequations}
\end{theorem} 

\begin{proof}{Proof. }
We have the following inner maximization problem  
\begin{align*}
    \max_{\boldsymbol{Z} \in \mathcal{Z}_1} g_3(\boldsymbol{x},\boldsymbol{U},\boldsymbol{Z}),
\end{align*}
where the objective is
\begin{align*}
  g_3(\boldsymbol{x},\boldsymbol{U},\boldsymbol{Z}) &=  \sum_{i=1}^n \left \langle \boldsymbol{P}^i, \frac{1}{x_i} \boldsymbol{U}_i \boldsymbol{U}_i^T \right \rangle + \left \langle \boldsymbol{Q}^i, \frac{1}{1-x_i} \left(\boldsymbol{x}-\boldsymbol{U}_i \right) \left( \boldsymbol{x} - \boldsymbol{U}_i\right)^T \right \rangle  \\ 
  & \quad \quad + \boldsymbol{U}_i^T \boldsymbol{r}^i + x_i w^i + \left(\boldsymbol{x}-\boldsymbol{U}_i \right)^T \boldsymbol{f}^i + (1-x_i)g^i \\ 
  &\quad \quad + \left \langle \boldsymbol{P}', \boldsymbol{x} \boldsymbol{x}^T \right \rangle + \boldsymbol{x}^T \boldsymbol{r}' + w' .
\end{align*} 
From Theorem 13.1 in \cite{bertsimas22}, we obtain the final formulation, in which both the objective and constraints \ref{eq:best_slc_degree3_a} - \ref{eq:best_slc_degree3_i} are from the dual of the problem $\max_{\boldsymbol{Z} \in \mathcal{Z}_1} g_3(\boldsymbol{x},\boldsymbol{U},\boldsymbol{Z})$. Moreover, from Theorem 13.1 in \cite{bertsimas22}, it follows that the extra variables $\boldsymbol{Y}^i, \boldsymbol{R}^i$ need to satisfy the additional LMIs 
\begin{align*}
    \begin{pmatrix}
                \boldsymbol{Y}^i & \boldsymbol{U}_i \\
                \boldsymbol{U}_i^T & x_i
           \end{pmatrix} \succeq \boldsymbol{0}, \quad  \begin{pmatrix}
                \boldsymbol{R}^i & \boldsymbol{x}-\boldsymbol{U}_i \\
                \left(\boldsymbol{x}-\boldsymbol{U}_i\right)^T & 1-x_i
           \end{pmatrix} \succeq \boldsymbol{0}, \quad \begin{pmatrix}
                \boldsymbol{Y}^{n+1} & \boldsymbol{x} \\
                \boldsymbol{x}^T & 1
           \end{pmatrix} \succeq \boldsymbol{0}.
\end{align*} \hfill $\square$
\end{proof}  

As will appear later in the numerical experiments, Theorem \eqref{thm:reformulation_degree3} allows us to derive high quality lower bounds, which in many cases are equal to the optimal value.

An alternative way to ensure convexity is from the Gershgorin Theorem \cite{gershgorin1931uber}, which states that a sufficient condition for convexity is that the Hessian matrix is diagonally dominant. Therefore, we can replace the LMIs with the following linear inequalities
\begin{align*}
    P^i_{jj} \ge \sum_{k \neq j} \left \lvert P^i_{jk} \right \rvert, \ Q^i_{jj} \ge \sum_{k \neq j} \left \lvert Q^i_{jk} \right \rvert, \ P_{jj}' \ge \sum_{k \neq j} \left \lvert P{jk}' \right \rvert, \quad i,j \in [n].
\end{align*} 

{\color{black} Although this alternative decreases the number of LMIs in the resulting problem, see Theorem \ref{thm:gersh_degree3} in Appendix A, it leads to much worse bounds than Theorem \ref{thm:reformulation_degree3}.}

One question that arises, is that if we have additional linear constraints in the feasible region, then does it add value to include those as well in the SLC decomposition. The answer in this case is no, as stated in the following theorem.

\begin{theorem}
    Assume that $\mathcal{X} \subseteq [0,1]^n$ and that we also have a linear inequality $d - \boldsymbol{c}^T \boldsymbol{x} \ge 0$. Then, {\color{black} adding the extra term $(d-\boldsymbol{c}^T\boldsymbol{x})z_2(\boldsymbol{x})$ in the SLC decomposition \eqref{eq:poly3_decomp_gen}}, where $z_2$ is a convex quadratic, does not improve the objective of the best SLC decomposition. 
\end{theorem} 

\begin{proof}{Proof. }
First, observe that the additional linear inequality can be written as 
\begin{align*}
    d - \boldsymbol{c}^T \boldsymbol{x} = \sum_{i=1}^n \lambda_i x_i + \sum_{i=1}^n \theta_i (1-x_i).
\end{align*} 
We can take $\lambda_i = -c_i, \theta_i = 0, \: i \in [n-1]$ and $\lambda_n = -c_n + d, \theta_n = d$. It follows that 
\begin{align*}
   & \sum_{i=1}^n x_i p_2^i(\boldsymbol{x}) + \sum_{i=1}^n(1-x_i) q_2^i(\boldsymbol{x}) + \beta_2(\boldsymbol{x}) + (d-\boldsymbol{c}^T \boldsymbol{x}) z_2(\boldsymbol{x})  \\
   &= \sum_{i=1}^n x_i p_2^i(\boldsymbol{x}) + \sum_{i=1}^n x_i \left(\lambda_i z_2(\boldsymbol{x}) \right) + \sum_{i=1}^n(1-x_i) q_2^i(\boldsymbol{x}) + \sum_{i=1}^n (1-x_i) (\theta_i z_2(\boldsymbol{x})) + \beta_2(\boldsymbol{x}) \\ 
    &= \sum_{i=1}^n x_i \overline{p}_2^i + \sum_{i=1}^n (1-x_i) \overline{q}_2^i(\boldsymbol{x}) + \beta_2(\boldsymbol{x}),
\end{align*} 
where $\overline{p}_2^i = p_2^i(\boldsymbol{x}) + \lambda_i z_2(\boldsymbol{x})$ and $\overline{q}_2^i(\boldsymbol{x}) = q_2^i(\boldsymbol{x}) + \theta_i z_2(\boldsymbol{x}) $. We can make the polynomials $\overline{p}_2^i, \overline{q}_2^i$ convex by transforming them as $\overline{p}_2^i(\boldsymbol{x}) = \overline{p}_2^i(\boldsymbol{x}) + \alpha \|\boldsymbol{x}\|^2 - \alpha n x_i$ and $\overline{q}_2^i(\boldsymbol{x}) = \overline{q}_2^i(\boldsymbol{x}) + \alpha \|\boldsymbol{x}\|^2$. Observe that $\sum_{i=1}^n x_i \overline{p}_2^i(\boldsymbol{x}) + \sum_{i=1}^n(1-x_i) \overline{q}_2^i(\boldsymbol{x}) = \sum_{i=1}^n x_i p_2^i(\boldsymbol{x}) + \sum_{i=1}^n (1-x_i) q_2^i(\boldsymbol{x})$. Therefore, since adding the extra term results in a decomposition that belongs in the proposed class and we are optimizing the best decomposition of the proposed class, the result follows. \hfill $\square$
\end{proof}  
In the next section, we deal with polynomials of degree 4.

\section{Polynomials of degree 4}   
We consider a generic polynomial of degree 4, that is,  
\begin{align*}
    p_4(\boldsymbol{x}) = \sum_{i \le j \le k \le l} c_{ijkl}^4 x_i x_j x_k x_l + \sum_{i \le j \le k} c_{ijk}^3 x_i x_j x_k + \boldsymbol{x}^T \boldsymbol{c}^2 \boldsymbol{x} + \boldsymbol{x}^T \boldsymbol{c}^1 + c^0,
\end{align*}    
where $\boldsymbol{c}^4 \in \mathbb{R}^{n \times n \times n \times n}, \boldsymbol{c}^3 \in \mathbb{R}^{n \times n \times n}, \boldsymbol{c}^2 \in \mathbb{R}^{n \times n}, \boldsymbol{c}^1 \in \mathbb{R}^n, c^0 \in \mathbb{R}$. We next discuss how to derive valid SLC decompositions for $p_4$, where the linear terms are defined by linearized products of two linear constraints, and the convex terms are polynomials of degree 2.  

\subsection{Existence of SLC decomposition}
We assume that $\mathcal{X} \subseteq [0,1]^n$. In this case, we will show that we can write $p_4(\boldsymbol{x})$ as the sum of quadratic functions, defined by $x_i x_j, \: x_i(1-x_j), \: (1-x_i)(1-x_j)$ for $i,j \in [n]$, multiplied by convex polynomials of degree 2, that is,
\begin{align}
    p_4(\boldsymbol{x}) &= \sum_{i \le j=1}^n x_i x_j p_2^{ij}(\boldsymbol{x}) + \sum_{i,j=1}^n x_i(1-x_j)q_2^{ij}(\boldsymbol{x}) + \sum_{i \le j=1}^n (1-x_i)(1-x_j) t_2^{ij}(\boldsymbol{x}), \\
    &\quad + \sum_{i=1}^n x_i p_2^i(\boldsymbol{x}) + \sum_{i=1}^n (1-x_i) q_2^i(\boldsymbol{x}) + \beta_2(\boldsymbol{x}), 
    \label{eq:poly4_decomp_gen}
\end{align} 
where  
\begin{align*}
    & p_2^{ij}(\boldsymbol{x}) = \boldsymbol{x}^T \boldsymbol{P}^{ij} \boldsymbol{x} + \boldsymbol{x}^T \boldsymbol{r}^{ij} + w^{ij}, \\ 
    & q_2^{ij}(\boldsymbol{x}) = \boldsymbol{x}^T \boldsymbol{Q}^{ij} \boldsymbol{x} + \boldsymbol{x}^T \boldsymbol{f}^{ij} + g^{ij}, \\ 
    & t_2^{ij}(\boldsymbol{x}) = \boldsymbol{x}^T \boldsymbol{T}^{ij} \boldsymbol{x} + \boldsymbol{x}^T \boldsymbol{h}^{ij} + s^{ij}, \\ 
    & p_2^i(\boldsymbol{x}) = \boldsymbol{x}^T \boldsymbol{P}^i \boldsymbol{x} + \boldsymbol{x}^T \boldsymbol{r}^i + w^i, \\ 
    & q_2^i(\boldsymbol{x}) = \boldsymbol{x}^T \boldsymbol{Q}^i \boldsymbol{x} + \boldsymbol{x}^T \boldsymbol{f}^i + g^i, \\ 
    & \beta_2(\boldsymbol{x}) = \boldsymbol{x}^T \boldsymbol{P}' \boldsymbol{x} + \boldsymbol{x}^T \boldsymbol{r}' + w',
\end{align*} 
and all polynomials $p_2^{ij}(\boldsymbol{x}), q_2^{ij}(\boldsymbol{x}), t_2^{ij}(\boldsymbol{x}),p_2^i(\boldsymbol{x}),q_2^i(\boldsymbol{x}),\beta_2(\boldsymbol{x})$ are convex. We have the following result.
\begin{theorem}
\label{thm:exist_degree4}
Every polynomial of degree 4, denoted as $p_4(\boldsymbol{x})$, can be written as 
\begin{align*}
    p_4(\boldsymbol{x}) &= \sum_{i \le j=1}^n x_i x_j p_2^{ij}(\boldsymbol{x}) + \sum_{i,j=1}^n x_i(1-x_j)q_2^{ij}(\boldsymbol{x}) + \sum_{i \le j=1}^n (1-x_i)(1-x_j) t_2^{ij}(\boldsymbol{x}), \\ 
    &\quad + \sum_{i=1}^n x_i p_2^i(\boldsymbol{x}) + \sum_{i=1}^n (1-x_i) q_2^i(\boldsymbol{x}) + \beta_2(\boldsymbol{x}),
\end{align*}  
where $p_2^{ij}(\boldsymbol{x}),q_2^{ij}(\boldsymbol{x}), t_2^{ij}(\boldsymbol{x}),p_2^{i}(\boldsymbol{x}),q_2^{i}(\boldsymbol{x}), \beta_2(\boldsymbol{x})$ are convex polynomials of degree 2, for all $i,j \in [n]$. 
\end{theorem}  

\begin{proof}{Proof. }
{\color{black} We define the matrices $F_{kl}^{ij} = c_{ijkl}^4, \: F_{jk}^i = c_{ijk}^3$ for all $i \le j \le k \le l \in [n]$. Moreover, we take $\boldsymbol{B} = \boldsymbol{c}^2, \: \boldsymbol{b} = \boldsymbol{c}^1, \: \eta = c^{0}$. We consider the degree 2 polynomials 
\begin{align*}
    & p_2^{ij}(\boldsymbol{x}) = \boldsymbol{x}^T \boldsymbol{F}^{ij} \boldsymbol{x} + \alpha \|\boldsymbol{x}\|^2, \quad i \le j \in [n], \\ 
    & q_2^{ij}(\boldsymbol{x}) = \alpha \|\boldsymbol{x}\|^2, \quad i \neq j \in [n], \\ 
    & t_2^{ij}(\boldsymbol{x}) = \alpha \|\boldsymbol{x}\|^2 , \quad i \le j \in [n], \\ 
    & p_2^i(\boldsymbol{x}) = \boldsymbol{x}^T \boldsymbol{F}^i \boldsymbol{x} + \alpha \|\boldsymbol{x}\|^2 - \left(1 + n + n(n+1)/2 \right)\alpha x_i, \quad i \in [n], \\ 
    & q_2^i(\boldsymbol{x}) = \alpha \|\boldsymbol{x}\|^2, \quad i \in [n], \\ 
    & \beta_2(\boldsymbol{x}) = \boldsymbol{x}^T \boldsymbol{B} \boldsymbol{x} + \boldsymbol{x}^T \boldsymbol{b} + \eta + \alpha \|\boldsymbol{x}\|^2,
\end{align*} 
where $\alpha_p^{ij} = \max_k \left\{ \sum_{l \neq k} |F_{kl}^{ij}| - F_{kk}^{ij} \right\}$ and $\alpha_p^i, \alpha_{\beta}$ are the same as in the proof of Theorem \ref{thm:exist_degree3}. We then take $\alpha = \max \left(0, \: \max_{i,j} \alpha_p^{ij}, \: \max_i \alpha_p^i, \: \alpha_{\beta} \right)$. By definition, the Hessian matrices of the quadratics are diagonally dominant, and therefore from the Gershgorin Theorem \cite{gershgorin1931uber}, it follows that they are positive semi-definite. Therefore, the polynomials $p_2^{ij},q_2^{ij},t_2^{ij},p_2^i,q_2^i,\beta_2$ are convex.} Further, they satisfy 
\begin{align*}
    & \sum_{i \le j=1}^n x_i x_j p_2^{ij}(\boldsymbol{x}) + \sum_{i,j=1}^n x_i(1-x_j) q_2^{ij}(\boldsymbol{x}) + \sum_{i \le j=1}^n (1-x_i)(1-x_j)t_2^{ij}(\boldsymbol{x}) \\
    &\quad + \sum_{i=1}^n p_2^i(\boldsymbol{x}) + \sum_{i=1}^n(1-x_i)q_2^i(\boldsymbol{x}) + \beta_2(\boldsymbol{x}) \: \: = \: \: p_4(\boldsymbol{x}).
\end{align*}  \hfill $\square$  
\end{proof}     

\begin{remark}
    We note that the result of Theorem \ref{thm:exist_degree4} also holds, if we only consider the terms in the objective, defined by $x_ix_j, \: x_i(1-x_j), \: (1-x_i)(1-x_j), \: i,j \in [n]$. However, by considering the additional terms, defined by $x_i, \: (1-x_i),\: 1, \: i \in [n]$, we obtain a stronger approximation.
\end{remark}

\subsection{Deriving the best SLC decomposition}
Observe that all valid SLC decompositions of Theorem \ref{thm:exist_degree4} are parametrized by coefficients, denoted as $\boldsymbol{Z}$, that satisfy certain equality  constraints, those of matching the coefficients of $p_4(\boldsymbol{x})$. In addition to the equality constraints, we also require that the degree 2 polynomials $p_2^{ij}(\boldsymbol{x}),q_2^{ij}(\boldsymbol{x}),t_2^{ij}(\boldsymbol{x})$ are convex, by imposing LMIs. Therefore, we can characterize all valid SLC decomposition of $p_4$ with the set $\mathcal{V}$, that can be written in compact form as follows:  
\begin{equation*}
    \mathcal{V} = \left\{ \begin{array}{c} 
    \boldsymbol{P}^{ij}, \boldsymbol{Q}^{ij}, \boldsymbol{T}^{ij} \in \mathbb{R}^{n \times n}, \\ 
     \boldsymbol{P}^{i}, \boldsymbol{Q}^{i}, \boldsymbol{P}' \in \mathbb{R}^{n \times n}, \\
    \boldsymbol{r}^{ij}, \boldsymbol{f}^{ij}, \boldsymbol{h}^{ij} \in \mathbb{R}^n, \\
    \boldsymbol{r}^{i}, \boldsymbol{f}^{i}, \boldsymbol{r}' \in \mathbb{R}^n, \\ 
    w^{ij}, g^{ij}, s^{ij} \in \mathbb{R}, \\
    w^{i}, g^{i}, w' \in \mathbb{R}, \end{array} :  \begin{array}{lr}
	 \sum_{i,j=1}^n \left \langle \overline{\boldsymbol{A}}_l^{ij}, \boldsymbol{P}^{ij} \right \rangle + \left \langle \overline{\boldsymbol{B}}_l^{ij}, \boldsymbol{Q}^{ij} \right \rangle + \left \langle \boldsymbol{C}_l^{ij}, \boldsymbol{T}^{ij} \right \rangle \\ 
     + \left( \overline{\boldsymbol{c}}_l^{ij} \right)^T \boldsymbol{r}^{ij}  + \left( \overline{\boldsymbol{d}}_l^{ij} \right)^T \boldsymbol{f}^{ij} + \left( \boldsymbol{e}_l^{ij} \right)^T \boldsymbol{h}^{ij} \\ 
     + \overline{\mu}_l^{ij} w^{ij} + \overline{\nu}_l^{ij} g^{ij} + \xi_l^{ij} \\ 
     + \sum_{i=1}^n \left \langle \boldsymbol{A}_l^{i}, \boldsymbol{P}^{i} \right \rangle + \left \langle \boldsymbol{B}_l^{i}, \boldsymbol{Q}^{i} \right \rangle + \left(\boldsymbol{c}_l^i\right)^T \boldsymbol{r}^i  \\ 
      + \left(\boldsymbol{d}_l^i\right)^T \boldsymbol{f}^i + + \mu_l^{i} w^{i} + \nu_l^{i} g^{i} \\ 
     + \left \langle \boldsymbol{\Xi}_l, \boldsymbol{P}' \right \rangle + \left(\boldsymbol{\omega}_l \right)^T \boldsymbol{r}' + \gamma_l w' = q_l, \: l \in [m], \\ 
    \boldsymbol{P}^{ij}, \boldsymbol{Q}^{ij}, \boldsymbol{T}^{ij}, \boldsymbol{P}^i, \boldsymbol{Q}^i, \boldsymbol{P}' \succeq \boldsymbol{0}, \ i \le j \in [n].
	 \end{array}  \right\}
\end{equation*}  

{\color{black} In the definition of the set $\mathcal{V}$, we use the notation $\overline{\boldsymbol{A}}_l^{ij}, \overline{\boldsymbol{B}}_l^{ij}, \boldsymbol{C}_l^{ij}, \boldsymbol{A}_l^{i}, \boldsymbol{B}_l^{i}, \boldsymbol{\Xi}_l, \overline{\boldsymbol{c}}_l^{ij}, \overline{\boldsymbol{d}}_l^{ij},$ $\boldsymbol{\omega}_l, \overline{\mu}_l^{ij}, \overline{\nu}_l^{ij}, \xi_l^{ij}, \mu_l^{i}, \nu_l^{i}, \gamma_l$ for the coefficients that multiply $\boldsymbol{P}^{ij},\boldsymbol{Q}^{ij},\boldsymbol{T}^{ij},\boldsymbol{P}^i,\boldsymbol{Q}^i,\boldsymbol{P}',\boldsymbol{r}^{ij},$ $\boldsymbol{f}^{ij},\boldsymbol{h}^{ij},\boldsymbol{r}^i,\boldsymbol{f}^i,\boldsymbol{r}',w^{ij},g^{ij},s^{ij},w^i,g^i,w'$, respectively, in the $l$-th constraint, while $q_l$ denotes the right-hand side.}

{\color{black} After obtaining an SLC decomposition, we apply RPT, while linearizing $\boldsymbol{U} = \boldsymbol{x} \boldsymbol{x}^T$, to obtain the following degree 3 terms in the objective objective
\begin{align*}
    & \sum_{i \le j=1}^n u_{ij} \left( \boldsymbol{x}^T \boldsymbol{P}^{ij} \boldsymbol{x} + \boldsymbol{x}^T \boldsymbol{r}^{ij} + w^{ij} \right) + \sum_{i,j=1}^n (x_i-u_{ij}) \left( \boldsymbol{x}^T \boldsymbol{Q}^{ij} \boldsymbol{x} + \boldsymbol{x}^T \boldsymbol{f}^{ij} + g^{ij} \right) \\ 
    &\quad \quad + \sum_{i \le j=1}^n (1-x_i-x_j+u_{ij}) \left( \boldsymbol{x}^T \boldsymbol{T}^{ij} \boldsymbol{x} + \boldsymbol{x}^T \boldsymbol{h}^{ij} + s^{ij} \right).
\end{align*}  
We further linearize $\boldsymbol{v}_{ij} = u_{ij} \boldsymbol{x}$. Further, we remove the nonconvex equality constraints $\boldsymbol{U} = \boldsymbol{x} \boldsymbol{x}^T, \: \: \boldsymbol{v}_{ij} = u_{ij} \boldsymbol{x} $, and try to achieve them by multiplying constraints in the feasible region and linearizing products with new variables, see \cite{bertsimascone}. We denote the new set of constraints as $\overline{\mathcal{X}}$. As a result, we obtain the following convex relaxation 
\begin{align*}
    g_4(\boldsymbol{x},\boldsymbol{U},\boldsymbol{V},\boldsymbol{Z}) &= g_3(\boldsymbol{x},\boldsymbol{U},\boldsymbol{Z}) + \sum_{i \le j =1}^n \frac{1}{u_{ij}} \boldsymbol{v}_{ij}^T \boldsymbol{P}^{ij} \boldsymbol{v}_{ij} + \boldsymbol{v}_{ij}^T \boldsymbol{r}^{ij} + u_{ij} w^{ij} \\  
    &+ \sum_{i,j=1}^n \frac{1}{x_i-u_{ij}} \left( \boldsymbol{U}_i - \boldsymbol{v}_{ij} \right)^T \boldsymbol{Q}^{ij} \left( \boldsymbol{U}_i - \boldsymbol{v}_{ij} \right) + \left( \boldsymbol{U}_i - \boldsymbol{v}_{ij} \right)^T \boldsymbol{f}^{ij} \\ 
    &\quad \quad \quad \quad + (x_i-u_{ij})g^{ij} \\ 
    &+ \sum_{i \le j=1}^n \frac{1}{1-x_i-x_j+u_{ij}} \boldsymbol{\theta}_{ij}^T \boldsymbol{T}^{ij} \boldsymbol{\theta}_{ij} +  \boldsymbol{\theta}_{ij}^T \boldsymbol{h}^{ij} \\ 
    &\quad \quad \quad \quad + (1-x_i-x_j+u_{ij})g^{ij}, 
\end{align*}    
where $\boldsymbol{\theta}_{ij} = \boldsymbol{x} - \boldsymbol{U}_i - \boldsymbol{U}_j + \boldsymbol{v}_{ij}$. Observe that $g_4$ is linear in the coefficients $\boldsymbol{Z}$ and convex in the problem variables $\left(\boldsymbol{x},\boldsymbol{U},\boldsymbol{V}\right)$, since it is defined as convex quadratics divided by linear functions. We follow the same reasoning as in the case of degree 3, that is, viewing the problem as a robust optimization problem and trying to optimize for all realization $\boldsymbol{Z} \in \mathcal{V}$. As a result, we obtain the following problem
\begin{equation}
       \min_{\left(\boldsymbol{x},\boldsymbol{U},\boldsymbol{V}\right) \in \overline{\mathcal{X}} } \: \max_{\boldsymbol{Z} \in \mathcal{V}} \: g_4(\boldsymbol{x},\boldsymbol{U},\boldsymbol{V},\boldsymbol{Z}). \label{eq:minmaxproblem_degree4}
\end{equation} 
Observe that in the above problem, we find the tightest lower bound out of the many possible ones. By using duality, we next show that the above problem can be reformulated as a semi-definite optimization problem.} We have the following result.
\begin{theorem}
\label{thm:reformulation_degree4}
   {\color{black} $\left(\boldsymbol{x},\boldsymbol{U},\boldsymbol{V}\right)$ is an optimal solution for Problem \eqref{eq:minmaxproblem_degree4} if and only if $\exists \left( \boldsymbol{Y}, \boldsymbol{R}, \boldsymbol{E}, \boldsymbol{\lambda} \right)$ such that $\left(\boldsymbol{x},\boldsymbol{U},\boldsymbol{V},\boldsymbol{Y},\boldsymbol{R},\boldsymbol{E},\boldsymbol{\lambda}\right)$ is an optimal solution for the following problem:}
   \begin{subequations}
   \begin{align}
  \displaystyle \min_{\boldsymbol{x}, \boldsymbol{U}, \boldsymbol{V}, \boldsymbol{Y}, \boldsymbol{R}, \boldsymbol{E}, \boldsymbol{\lambda}} &\quad - \sum_{l=1}^{m} \lambda_l q_l \nonumber \\
     {\rm s.t.} &\quad \eqref{eq:best_slc_degree3_a} - \eqref{eq:best_slc_degree3_l} \nonumber \\ 
        &\quad \boldsymbol{Y}^{ij} + \sum_{l=1}^{m} \lambda_l \overline{\boldsymbol{A}}_l^{ij} \preceq \boldsymbol{0}, \quad i \le j \in [n], \label{eq:best_slc_degree4_a}  \\
        &\quad  \boldsymbol{R}^{ij} + \sum_{l=1}^{m} \lambda_l \overline{\boldsymbol{B}}_l^{ij} \preceq \boldsymbol{0}, \quad i,j \in [n], \label{eq:best_slc_degree4_b}  \\
        &\quad  \boldsymbol{E}^{ij} + \sum_{l=1}^{m} \lambda_l \boldsymbol{C}_l^{ij} \preceq \boldsymbol{0}, \quad i \le j \in [n], \label{eq:best_slc_degree4_c} \\ 
        &\quad \boldsymbol{Y}^{i} + \sum_{l=1}^{m} \lambda_l \boldsymbol{A}_l^{i} \preceq \boldsymbol{0}, \quad i \in [n], \label{eq:best_slc_degree4_d} \\
        &\quad \boldsymbol{R}^{i} + \sum_{l=1}^{m} \lambda_l \boldsymbol{B}_l^{i} \preceq \boldsymbol{0}, \quad i \in [n], \label{eq:best_slc_degree4_e} \\ 
        &\quad \boldsymbol{Y}^{n+1} + \sum_{l=1}^{m} \lambda_l \boldsymbol{\Xi}_l \preceq \boldsymbol{0}, \label{eq:best_slc_degree4_f} \\
        &\quad \boldsymbol{v}_{ij} + \sum_{l=1}^{m} \lambda_{l} \boldsymbol{c}_l^{ij} = \boldsymbol{0}, \quad i \le j \in [n], \label{eq:best_slc_degree4_g}  \\ 
        &\quad \boldsymbol{U}_i - \boldsymbol{v}_{ij} + \sum_{l=1}^{m} \lambda_{l} \boldsymbol{d}_l^{ij} = \boldsymbol{0}, \quad i, j \in [n], \label{eq:best_slc_degree4_h}  \\
        &\quad \boldsymbol{x} - \boldsymbol{U}_i - \boldsymbol{U}_j + \boldsymbol{v}_{ij} + \sum_{l=1}^{m} \lambda_{l} \boldsymbol{e}_l^{ij} = \boldsymbol{0}, \quad i \le j \in [n], \label{eq:best_slc_degree4_i}  \\
        &\quad u_{ij} + \sum_{l=1}^{m} \lambda_l \mu_l^{ij} = 0, \quad i \le j \in [n], \label{eq:best_slc_degree4_j} \\ 
        &\quad x_i-u_{ij} + \sum_{l=1}^{m} \lambda_l \nu_l^{ij} = 0, \quad i,j \in [n], \label{eq:best_slc_degree4_k} \\  
        &\quad 1-x_i-x_j+u_{ij} + \sum_{l=1}^{m} \lambda_l \xi_l^{ij} = 0, \quad i \le j \in [n], \label{eq:best_slc_degree4_l} \\  
        &\quad \begin{pmatrix}
                \boldsymbol{Y}^{ij} & \boldsymbol{v}_{ij} \\
                \left(\boldsymbol{v}_{ij}\right)^T & u_{ij}
           \end{pmatrix} \succeq \boldsymbol{0}, \quad i \le j \in [n], \label{eq:best_slc_degree4_m} \\ 
       &\quad \begin{pmatrix}
                \boldsymbol{R}^{ij} & \boldsymbol{U}_i-\boldsymbol{v}_{ij} \\
                \left(\boldsymbol{U}_i-\boldsymbol{v}_{ij}\right)^T & x_i-u_{ij}
           \end{pmatrix} \succeq \boldsymbol{0}, \quad i,j \in [n], \label{eq:best_slc_degree4_n} \\ 
        &\quad \begin{pmatrix}
                \boldsymbol{E}^{ij} & \boldsymbol{x}-\boldsymbol{U}_i-\boldsymbol{U}_j+\boldsymbol{v}_{ij} \\
                \left( \boldsymbol{x}-\boldsymbol{U}_i-\boldsymbol{U}_j+\boldsymbol{v}_{ij} \right)^T & 1- x_i -x_j + u_{ij}
           \end{pmatrix} \succeq \boldsymbol{0}, \: i \le j \in [n], \label{eq:best_slc_degree4_l} \\ 
      &\quad \left( \boldsymbol{x}, \boldsymbol{U}, \boldsymbol{V} \right) \in \overline{\mathcal{X}}. \label{eq:best_slc_degree4_o}
			\end{align}
\end{subequations}
\end{theorem}  

\begin{proof}{Proof. }
We have the following inner maximization problem  
\begin{align*}
    \max_{\boldsymbol{Z} \in \mathcal{V}} g_4(\boldsymbol{x},\boldsymbol{U}, \boldsymbol{V},\boldsymbol{Z}),
\end{align*}
where the objective is
\begin{align*}
  g_4(\boldsymbol{x},\boldsymbol{U},\boldsymbol{V},\boldsymbol{Z}) &=  \sum_{i \le j=1}^n \left \langle \boldsymbol{P}^{ij}, \frac{1}{u_{ij}} \boldsymbol{v}_{ij} \boldsymbol{v}_{ij}^T \right \rangle + \boldsymbol{v}_{ij}^T \boldsymbol{r}^{ij} + u_{ij} w^{ij} \\ 
  &+ \sum_{i,j=1}^n \left \langle \boldsymbol{Q}^{ij}, \frac{1}{x_i-u_{ij}} \left(\boldsymbol{U}_i-\boldsymbol{v}_{ij} \right) \left( \boldsymbol{U}_i - \boldsymbol{v}_{ij} \right)^T \right \rangle \\ 
  &\quad \quad \quad + \left(\boldsymbol{U}_i-\boldsymbol{v}_{ij} \right)^T \boldsymbol{f}^{ij} + (x_i-u_{ij})g^{ij} \\ 
  &+ \sum_{i \le j=1}^n \left \langle \boldsymbol{T}^{ij}, \frac{1}{1-x_i-x_j+u_{ij}} \boldsymbol{\theta}_{ij} \boldsymbol{\theta}_{ij}^T \right \rangle \\ 
  &\quad \quad \quad + \boldsymbol{\theta}_{ij}^T \boldsymbol{h}^{ij} + (1-x_i-x_j+u_{ij})s^{ij},
\end{align*}  
where $\boldsymbol{\theta}_{ij} = \boldsymbol{x} - \boldsymbol{U}_i - \boldsymbol{U}_j + \boldsymbol{v}_{ij}$. First, we note that constraints \ref{eq:best_slc_degree3_a} - \ref{eq:best_slc_degree3_l} follow from Theorem \ref{thm:reformulation_degree3}. Moreover, from Theorem 13.1 in \cite{bertsimas22}, we obtain the final formulation, in which both the objective and constraints \ref{eq:best_slc_degree4_a} - \ref{eq:best_slc_degree4_l} are from the dual of the problem $\max_{\boldsymbol{Z} \in \mathcal{V}} g_4(\boldsymbol{x},\boldsymbol{U},\boldsymbol{V},\boldsymbol{Z})$. Moreover, from Theorem 13.1 in \cite{bertsimas22}, it follows that the extra variables $\boldsymbol{Y}^{ij}, \boldsymbol{R}^{ij}, \boldsymbol{E}^{ij}$ need to satisfy the additional LMIs 
\begin{align*}
    \begin{pmatrix}
                \boldsymbol{Y}^{ij} & \boldsymbol{v}_{ij} \\
                \left(\boldsymbol{v}_{ij}\right)^T & u_{ij}
           \end{pmatrix} \succeq \boldsymbol{0}, \quad  \begin{pmatrix}
                \boldsymbol{R}^{ij} & \boldsymbol{U}_i-\boldsymbol{v}_{ij} \\
                \left(\boldsymbol{U}_i-\boldsymbol{v}_{ij}\right)^T & x_i-u_{ij}
           \end{pmatrix} \succeq \boldsymbol{0}, \\  
           \quad \begin{pmatrix}
                \boldsymbol{E}^{ij} & \boldsymbol{x}-\boldsymbol{U}_i-\boldsymbol{U}_j+\boldsymbol{v}_{ij} \\
                \left( \boldsymbol{x}-\boldsymbol{U}_i-\boldsymbol{U}_j+\boldsymbol{v}_{ij} \right)^T & 1- x_i -x_j + u_{ij}
           \end{pmatrix} \succeq \boldsymbol{0}.
\end{align*} \hfill $\square$
\end{proof}  
We note that, similarly to the degree 3 case, an alternative way to ensure convexity is by ensuring that the Hessian matrices of the quadratics are diagonally dominant, see Appendix A for the formulation. However, this approach results in bounds that are much worse than Theorem \eqref{thm:reformulation_degree4}. 

\begin{remark}
For problems that involve polynomial constraints, we can apply the procedure outlined so far in order to derive the best SLC decomposition for each one of them.
\end{remark} 

\begin{remark}
    We note that one could also optimize for the best SLC decomposition with the adversarial approach \cite{mutapcic2009cutting}. Briefly, the adversarial approach alternates between solving a master problem, by including a finite number of scenarios $\boldsymbol{Z}$, and a subproblem, in which the worst case scenarios are computed. The procedure is repeated iteratively until the worst case scenarios do not violate constraints. One advantage of this approach is that the total number of LMIs is reduced, however it may require a large number of iterations to converge. 
\end{remark}

\section{Polynomials of arbitrary degree}  
\subsection{Existence of SLC decomposition}
{\color{black} In the more general case, we consider Problem \eqref{eq:genericproblem}, with $p_d$ being an arbitrary polynomial of degree $d$. We propose to obtain an SLC decomposition for $p_d$, by generalizing the one derived in the case of degree 4, see Section 3. More precisely, for a polynomial of arbitrary degree ($d$), we consider the SLC decomposition, in which the linear terms are defined as $d-2$ linearized products of the linear constraints of the feasible region, and the convex terms are polynomials of degree 2. We define the following functions that contain products of degree up to $d-2$
\begin{equation}
       f_{\mathcal{I},\mathcal{J}}(\boldsymbol{x}) = \prod_{i \in \mathcal{I}} x_i \prod_{j \in \mathcal{J}} (1-x_j), \label{eq:function_definition_f}
\end{equation} 
where $\mathcal{I}$ and $\mathcal{J}$ are disjoint subsets of $\{1,.,n\}$, such that $|\mathcal{I} \cup \mathcal{J}|$ is the degree of the polynomial. Observe that by appropriately selecting $\mathcal{I}$ and $\mathcal{J}$, we can cover all polynomial terms of degree up to $d-2$. We have the following result regarding the existence of this type of SLC decomposition.
\begin{theorem}
\label{thm:exists_degree_arb}
Assuming that $\mathcal{X} \subseteq [0,1]^n$, then every polynomial of degree $d$, denoted as $p_d(\boldsymbol{x})$, can be written as 
\begin{align*}
    p_d(\boldsymbol{x}) = \sum_{\substack{\mathcal{I}, \mathcal{J} \\ |\mathcal{I} \cup \mathcal{J}| \leq d-2 \\ \mathcal{I} \cap \mathcal{J} = \emptyset}} f_{\mathcal{I},\mathcal{J}}(\boldsymbol{x}) q_{\mathcal{I},\mathcal{J}}(\boldsymbol{x}),
\end{align*}   
where $f_{\mathcal{I},\mathcal{J}}(\boldsymbol{x})$ is defined as in \ref{eq:function_definition_f}, and $q_{\mathcal{I},\mathcal{J}}(\boldsymbol{x})$ are convex quadratics. 
\end{theorem} 

\begin{proof}{Proof. }
We can match the different degree terms by appropriately selecting the index set $\mathcal{I}$, while setting $\mathcal{J} = \emptyset$. For the terms of degree $d$, of the form $cx_{\sigma(1)} \cdots x_{\sigma(d)}$, we take $\mathcal{I} = \{\sigma(1),..,\sigma(d-2)\}$ and then define the quadratic $q_{\mathcal{I},\mathcal{J}}(\boldsymbol{x}) = \boldsymbol{x}^T \boldsymbol{P}^{\mathcal{I},\mathcal{J}}\boldsymbol{x}$, where $\boldsymbol{P}^{\mathcal{I},\mathcal{J}} = \boldsymbol{0}$, and $\boldsymbol{P}^{\mathcal{I},\mathcal{J}}_{\sigma(d-1) \sigma(d)} = c$. Moreover, for the terms of degree $d-1$, of the form $cx_{\sigma(1)} x_{\sigma(2)}...x_{\sigma(d-1)}$, we take the index set $\mathcal{I} = \{\sigma(1),..,\sigma(d-3)\}$ and then define the quadratic $q_{\mathcal{I},\mathcal{J}}(\boldsymbol{x}) = \boldsymbol{x}^T \boldsymbol{P}^{\mathcal{I},\mathcal{J}}\boldsymbol{x}$, where $\boldsymbol{P}^{\mathcal{I},\mathcal{J}} = \boldsymbol{0}$ and $\boldsymbol{P}^{\mathcal{I},\mathcal{J}}_{\sigma(d-2) \sigma(d-1)} = c$. Similarly, we can match the terms of degree $d-2$ by taking index sets $\mathcal{I}$, such that $|\mathcal{I}| = d-4$ and repeat the process all the way to degree $4$, in which case we can define the remaining functions as in Theorem \ref{thm:exist_degree4}. Moreover, for each quadratic $q_{\mathcal{I},\mathcal{J}}(\boldsymbol{x})$, we define $\alpha^{\mathcal{I},\mathcal{J}} = \max_k \left\{ \sum_{l \neq k} |P_{kl}^{\mathcal{I},\mathcal{J}}| - P_{kk}^{\mathcal{I},\mathcal{J}} \right\}$. Then, we take $\alpha = \max \left(0, \max_{\mathcal{I},\mathcal{J}} \alpha^{\mathcal{I},\mathcal{J}} \right)$ and add the term $\alpha \|\boldsymbol{x}\|^2$ to each quadratic. By definition, the Hessian matrices of the quadratics are diagonally dominant, and therefore from the Gershgorin Theorem \cite{gershgorin1931uber}, it follows that they are positive semi-definite. Therefore, the quadratics $q_{\mathcal{I},\mathcal{J}}$ are convex. Further, we can add additional linear terms in the quadratics, without affecting convexity, in order to cancel out the extra terms arising from the addition of $\alpha\|\boldsymbol{x}\|^2$, see Theorem \ref{thm:exist_degree4}. As a result, we obtain that
\begin{align*}
   \sum_{\substack{\mathcal{I}, \mathcal{J} \\ |\mathcal{I} \cup \mathcal{J}| \leq d-2 \\ \mathcal{I} \cap \mathcal{J} = \emptyset}} f_{\mathcal{I},\mathcal{J}}(\boldsymbol{x}) q_{\mathcal{I},\mathcal{J}}(\boldsymbol{x})  \: = \: p_d(\boldsymbol{x}).
\end{align*} \hfill $\square$   
\end{proof}  

\subsection{Deriving the best SLC decomposition} 
Observe that the SLC decompositions proposed in Theorem \ref{thm:exists_degree_arb} are parametrized by the coefficients of the quadratics. Therefore, we can apply the methodology described for polynomials of degree 3 and 4 to derive the best SLC decomposition. First, we linearize products in the functions $f_{\mathcal{I},\mathcal{J}}(\boldsymbol{x})$ and obtain the function $\Tilde{f}_{\mathcal{I},\mathcal{J}}(\boldsymbol{x},\boldsymbol{U})$, that is linear in both $\boldsymbol{x}$ and $\boldsymbol{U}$. We multiply the constraints in the feasible region to link $\boldsymbol{x}$ with $\boldsymbol{U}$ and denote the new feasible region with $\overline{X}$. Let $\boldsymbol{Z}$ denote the coefficients of the quadratics $q_{\mathcal{I},\mathcal{J}}(\boldsymbol{x})$ and $\mathcal{Z}$ denote the constraints that need to satisfy.  We consider the problem 
\begin{align*}
    \min_{\boldsymbol{x},\boldsymbol{U} \in \overline{X}} \max_{\boldsymbol{Z} \in \mathcal{Z}}  \sum_{\substack{\mathcal{I}, \mathcal{J} \\ |\mathcal{I} \cup \mathcal{J}| \leq d-2 \\ \mathcal{I} \cap \mathcal{J} = \emptyset}} \Tilde{f}_{\mathcal{I},\mathcal{J}}(\boldsymbol{x},\boldsymbol{U}) q_{\mathcal{I},\mathcal{J}}(\boldsymbol{x}). 
\end{align*} 
We then apply the same methodology as in Theorem \ref{thm:reformulation_degree4}, that is, derive the dual of the inner maximization problem and reformulate it as a minimization problem. The size of the largest LMI will be $\mathcal{O}(n^2)$, similarly as in Theorem \ref{thm:reformulation_degree4}, but the total number of variables will be increased to $\mathcal{O}(n^d)$. We note that in case $\mathcal{X}$ contains only conic constraints, then the problem resulting from the best SLC decomposition will be a conic optimization problem, for which there exist very efficient solvers, such as MOSEK \cite{mosek101}.

Finally, we note that an alternative SLC decomposition can be obtained by defining the linear terms as $x_i, \: 1-x_i$ for $i \in [n]$, and the convex terms as polynomials of degree $d-1$, see Appendix B for the proof of existence of such an SLC decomposition. However, if we define the SLC decomposition in this way, then we are not able to find the best SLC decomposition, and as a result we would obtain loose bounds.
}

\section{Comparison with other methods for polynomial optimization}   
In this section, we illustrate the pros and cons of our approach compared to two other methods for polynomial optimization, those are the SOS approach and RLT for polynomial optimization.
\subsection{Sum of squares for polynomial optimization}
The SOS method has been proposed in order to obtain a lower bound on the global minimum of an arbitrary polynomial of degree $d$, and it can be formulated as follows 
\begin{equation}
    \begin{array}{cll}
        \displaystyle \max_{\boldsymbol{x} \in \mathcal{X}} &\quad \alpha \\
        {\rm s.t.} 	&\quad p_d(\boldsymbol{x}) \ge \alpha,
    \end{array} \label{eq:sostools_1}
\end{equation}  
where the set $\mathcal{X}$ contains linear or polynomial constraints. The main idea is to replace the constraint $p_d(\boldsymbol{x}) - \alpha \ge 0$, with the constraint that $p_d(\boldsymbol{x}) - \alpha$ is a sum of squares of polynomials, see \cite{parrilo2001minimizing}. Then \cite{parrilo2003semidefinite} showed that a sufficient condition to satisfy the latter is 
\begin{align*}
  \exists \: \boldsymbol{Q} \succeq \boldsymbol{0} \quad \text{such that} \quad  p_d(\boldsymbol{x}) - \alpha \: = \: \boldsymbol{z}(\boldsymbol{x})^T \boldsymbol{Q} \boldsymbol{z}(\boldsymbol{x}),
\end{align*} 
where $\boldsymbol{z}(\boldsymbol{x})$ denotes the vector of all monomials of degree up to $d/2$. For example, if $p_4(x) = x^4+x^2+x$, then  we have $z(x) = [1,x,x^2]$. In the general case of a $n$-dimensional polynomial of even degree $d$, the dimension of $\boldsymbol{z}(\boldsymbol{x})$ will be $ \binom{n+d/2}{d/2}$, see \cite{parrilo2001minimizing}. Therefore, when optimizing an arbitrary polynomial of even degree $d$, the SOS relaxation contains an LMI of dimension $\binom{n+d/2}{d/2} \times \binom{n+d/2}{d/2}$, and has $\mathcal{O}(n^{d})$ variables. Moreover, in order to converge to the optimum the SOS method might need to consider polynomials of higher degree in the SOS decomposition, therefore leading to LMIs of even larger dimension, see \cite{parrilo2003semidefinite}.

As we have seen in this paper, our proposed approach also results in a problem that involves LMIs. However, the largest LMI is of dimension $(n+1) \times (n+1)$. Observe that if the degree of the polynomial is greater than 4, then the size of the largest LMI, will still be $(n+1) \times (n+1)$, as we can linearize all products with new variables and still have an SLC objective, in which the convex functions are polynomials of degree 2. Moreover, our method results in $\mathcal{O}(n^{d-1})$ variables and $\mathcal{O}(n^d)$ parameters for the SLC representation. We note that one might also use the adversarial approach to obtain the best representation. 
 
Moreover, we note that not all non-negative polynomials can be written as SOS. The existence of a SOS decomposition for a non-negative polynomial holds when $n=2$ or $d=2$ and $n=3,d=4$, see \cite{reznick2000some}. Moreover, Hilbert proved that if $n \ge 4$ and $d \ge 4$, then there exist non-negative polynomials that cannot be written as SOS. {\color{black} Therefore, the SOS method yields an approximation, while in some cases, a certificate is obtained by retrieving the optimal solution $\boldsymbol{x}^*$, and evaluating the original polynomial at the latter. Finally, we note that our method can also handle non-polynomial constraints, contrary to the SOS method.}

\subsection{RLT for polynomial optimization}  
Another approach for polynomial optimization problems, is the Reformulation-Linearization-Technique (RLT), introduced in \cite{sa90} and extended to polynomial optimization problems in \cite{sherali2013}. In this case, a linear relaxation of Problem \eqref{eq:genericproblem} is obtained by linearizing products of variables with new variables, for example, $X_{122} = x_1 x_2^2$. Then, in order to obtain bounds on the new variables the linear constraints of the feasible region are multiplied with each other. As a result, we obtain $\mathcal{O}(n^d)$ variables and $\mathcal{O}(n^d)$ constraints, while the obtained bound can be loose. More importantly, it has been noted that RLT is a special case of the SOS method, see \cite{parrilo2003semidefinite}, and as a result it is outperformed by the latter.

More recently, \cite{dalkiran2016rlt} introduced enhancements of the RLT approach, that allowed them to solve problems with 16 variables for polynomials of degree 3 and 12 variables for polynomials of degree 4. As we illustrate in Section 6, with our approach we can solve problems with $40$ variables and degree 3, as well as $20$ variables and degree 4. Moreover, our method also applies to problems with non-polynomial constraints, while the RLT approach is restricted to problems with linear or polynomial constraints.

{\color{black} Finally, we note that we can easily make a variant of our method, that is provably better than RLT. Suppose our polynomial is degree 3. Then, we can add all degree 4 terms with 0 coefficient. In this case, our approach contains RLT, but finds much better approximations.}

\section{Numerical Experiments}  
In this section, we demonstrate numerically the value of the proposed framework. We consider the following feasible regions 
\begin{align*}
	\begin{array}{cll}
	\displaystyle \mathcal{X}_1 &= \displaystyle \{\boldsymbol{x} \in \mathbb{R}^{n_x} \mid 0 \le x_i \le 1, \ i \in [n_x] \}, \\
	\displaystyle \mathcal{X}_2 &= \displaystyle \left\{ \boldsymbol{x} \in \mathcal{X}_1 \mid p_3^2(\boldsymbol{x}) \le 0 \right\}, \\ 
    \displaystyle \mathcal{X}_2' &= \left\{ \boldsymbol{x} \in \mathcal{X}_1 \mid p_4^2(\boldsymbol{x}) \le 0 \right\}, \\
       \displaystyle \mathcal{X}_3 &= \displaystyle \left\{ \boldsymbol{x} \in \mathcal{X}_1 \mid \log\left( \sum_{i=1}^n \exp(x_i) \right) \le \alpha \right\},
	\end{array}
\end{align*}  
where $p_3^2(\boldsymbol{x}),p_4^2(\boldsymbol{x})$ are polynomials of degree 3 and 4, respectively. We use the same branch and bound scheme for our method as in \cite{bertsimas2023novel}. The numerical experiments are performed on an Intel i9 2.3GHz CPU core with 16.0GB of RAM. The computations are implemented in Julia 1.5.3 and the Julia package \texttt{JuMP.jl} version 0.21.6. All computations for our approach are conducted with MOSEK 9.3.18 \cite{mosek101}. Further, all computations for the SOS method, are conducted with the SOSTools software \cite{prajna2002introducing} implemented in the Julia package \texttt{SumOfSquares} version 0.4.3, and all computations for BARON are conducted with BARON version 20.10.16 \cite{sahinidis1996baron} implemented using the Python package \texttt{pyomo} version 6.4.1. We use the abbreviations ``Opt", ``Hyp", and ``Time" to represent the optimal value, the total number of hyperplanes generated during branch and bound, and the computation time in seconds, respectively. Moreover, we set the maximum time limit equal to 3600 seconds, hence if the computation time equals $3600^*$, the optimum cannot be found within 3600 seconds and all approaches return the best value they can obtain within 3600 seconds. For each value of the dimension $n$, the results are averaged over 10 randomly generated instances. The coefficients of the polynomials are randomly generated in the interval $[-5,5]$, with a density of $0.5$ for degree 3, and $0.2$ for degree 4. Finally, for the set $\mathcal{X}_3$ we take $\alpha=3$ for $n \in \{10,15\}$ and $\alpha=3.5$ for $n \in \{20,30\}$.

\subsection{Polynomials of degree 3}
We first consider the case of degree 3 polynomials. We consider the following problem:
\begin{equation}
    \begin{array}{cll}
        \displaystyle \min_{\boldsymbol{x}} &\quad p_3(\boldsymbol{x}) \\
        {\rm s.t.} 	&\quad \boldsymbol{x} \in \mathcal{X}.
    \end{array} \label{eq:poly3_numexp_1}
\end{equation}     
{\color{black} First, in Table \ref{tab:results_degree3_decomp}, we illustrate that the best SLC decomposition finds the optimum in many cases and gives much better bounds than a random SLC decomposition.} Moreover, in Table \ref{tab:results_degree3}, we compare our method with SOS and BARON.

\begin{table}[t]   
    \centering
    \begin{tabular}{lclllllllll}
\hline
	\multirow{2}{*}{$\mathcal{X}$} & & \multirow{2}{*}{$n_x$} & &  \multicolumn{3}{c}{Best-SLC}  &  & \multicolumn{3}{c}{Random-SLC} \\ 
		\cline{5-7} \cline{9-11}
    	& & &  & LB & UB & Time(s)  &  & LB & UB & Time(s) \\ \hline \hline
         & & $10$ & & -69.84 & -69.84 & 0.51 & & -628.46 & -69.39 & 0.13 \\
	\multirow{2}{*}{$\mathcal{X}_1$} & & $20$ & & -308.90 & -308.90 & 17.11 & & -10956.91 & -306.88 & 1.84 \\
	 & & $30$ & & -710.05 & -710.05 & 303.43 & & -81439.45 & -709.89 & 15.66 \\
	\hline 
	& & $10$ & & -62.28 & -53.33 & 1.07 & & -628.46 & -0.74 & 0.22 \\
	\multirow{2}{*}{$\mathcal{X}_2$} & & $20$ & & -304.70 & -283.92 & 51.22 & & -10956.91 & -9.47 & 4.91 \\ 
	& & $30$ & & -708.62 & -631.08 & 630.58 & & -81439.45 & -32.90 & 30.38 \\
\hline
    \end{tabular}
    \caption{Comparison of obtained bounds for a random versus the "best" SLC decomposition. \label{tab:results_degree3_decomp}}
\end{table}   

\begin{table}[H]   
    \centering 
    \begin{tabular}{lcllllllllllllllll}
    \hline 
    \multirow{2}{*}{$\mathcal{X}$} & & \multirow{2}{*}{$n$} & &  \multicolumn{3}{c}{Ours} &  & \multicolumn{2}{c}{BARON} & & \multicolumn{2}{c}{SOS} \\
		\cline{5-7} \cline{9-10} \cline{12-13} 
    	& & &  & Opt & Time & Hyp &  & Opt & Time & & Opt & Time \\ \hline \hline 
        \multirow{4}{*}{$\mathcal{X}_1$} & & $10$ & & -69.84 & 0.51 & 0 & & -69.84 & 44.74 & & -69.84 & 0.15 \\
        & & $20$ & & -308.90 & 17.37 & 0 & & -308.88 & $3600^{*}$ & & -308.90 & 2.89 \\
        & & $30$ & & -710.05 & 314.32 & 0 & & -710.05 & $3600^{*}$ & & - & - \\ 
        & & $40$ & & -1248.90 & 2964.22 & 0 & & -1246.49 & $3600^{*}$ & & - & - \\
	\hline 
        \multirow{3}{*}{$\mathcal{X}_2$} & & $10$ & & -61.28 & 4.29 & 1.7 & & -61.28 & 232.18 & & $-62.26^{\dag}$ & 0.34 \\
        & & $20$ & & -303.24 & 177.01 & 1.5 & & -303.24 & $3600^{*}$ & & $-304.69^{\dag}$ & 4.34 \\
        & & $30$ & & -707.99 & 1680.28 & 0.8 & & -707.99 & $3600^{*}$ & & - & - \\ 
        \hline 
        \multirow{3}{*}{$\mathcal{X}_3$} & & $10$ & & -64.14 & 3.16 & 1.1 & & -64.14 & 184.09 & &  &  \\
        & & $20$ & & -170.84 & 729.13 & 8.1 & & -170.38 & $3600^{*}$ & &  &  \\
        & & $30$ & & -17.85 & $3600^{*}$ & 4.5 & & -17.85 & $3600^{*}$ & &  &  \\ 
        \hline 
    \end{tabular} 
    \caption{Results for polynomials of degree 3. A "$-$" on the SOS method indicates that it ran out of memory and a $\dag$ indicates that a lower bound was returned without a certificate of optimality. \label{tab:results_degree3}}
\end{table} 

From Table \ref{tab:results_degree3}, we observe that all instances for $\mathcal{X}_1$ were directly solved at the root node. Moreover, for $\mathcal{X}_2$ and $\mathcal{X}_3$ the obtained bounds were still very good but some additional branch and bound was needed in order to solve the problem to optimality. For $\mathcal{X}_1$, we notice that for $n=10$ and $n=20$ SOS was able to find the optimum very fast. However, we notice a significant improvement from our method over BARON and SOS in the larger scaler problems. More precisely, our approach could efficiently solve problems for $n=30$ and $n=40$. On the other hand, BARON was able to find the optimum in all instances for $n=30$, but could not prove optimality within one hour, while for $n=40$ it was not able to find the optimum within one hour for some instances and returned a greater value. Moreover, SOS ran out of memory for both $n=30$ and $n=40$.  For $ \mathcal{X}_2$, SOS finds a solution that is a lower bound for the optimum, in which case our method achieves the best performance. Finally, for $\mathcal{X}_3$, SOS is no longer applicable, while our method is able to solve the problems faster than BARON for $n=10$ and $n=20$. We notice that BARON was not able to find the optimum in some cases for $n=20$ and returned a larger value.

\subsection{Polynomials of degree 4}  
In this section, we apply our method to problems with polynomials of degree 4. We consider the following problem 
\begin{equation}
    \begin{array}{cll}
        \displaystyle \min_{\boldsymbol{x}} &\quad p_4(\boldsymbol{x}) \\
        {\rm s.t.} 	&\quad \boldsymbol{x} \in \mathcal{X}.
    \end{array} \label{eq:poly4_numexp}
\end{equation}     
The results for our approach in comparison with BARON and SOS, are illustrated in Table \ref{tab:results_degree4}. 

\begin{table}[H]   
    \centering 
    \begin{tabular}{lcllllllllllllllll}
    \hline 
    \multirow{2}{*}{$\mathcal{X}$} & & \multirow{2}{*}{$n$} & &  \multicolumn{3}{c}{Ours} &  & \multicolumn{2}{c}{BARON} & & \multicolumn{2}{c}{SOS} \\
		\cline{5-7} \cline{9-10} \cline{12-13} 
    	& & &  & Opt  & Time & Hyp &  & Opt & Time & & Opt & Time \\ \hline \hline 
        \multirow{3}{*}{$\mathcal{X}_1$} & & $10$ & & -144.29 & 11.17 & 0 & & -143.58 & $3600^{*}$ & & -144.29 & 9.16 \\
        & & $15$ & & -441.26 & 222.75 & 0 & & -435.71 & $3600^{*}$ & & -441.26 & 416.99 \\
        & & $20$ & & -1323.71 & 2001.45 & 0 & & 1319.33 & $3600^{*}$ & & - & - \\ 
	\hline 
        \multirow{2}{*}{$\mathcal{X}_2'$} & & $10$ & & -135.42 & 83.73 & 1.1 & & -135.42 & $3600^{*}$ & & $-138.33^{\dag}$ & 10.92 \\
        & & $15$ & & -423.67 & 1905.88 & 1.2 & & -423.67 & $3600^{*}$ & & $-430.98^{\dag}$ & 501.64 \\
        \hline 
        \multirow{2}{*}{$\mathcal{X}_3$} & & $10$ & & -136.33 & 18.04 & 0.2 & & -136.33 & $3600^{*}$  & &  &  \\
        & & $15$ & & -68.47 & 1707.86 & 3.1 & & -68.47 & $3600^{*}$ & &  &  \\
        \hline 
    \end{tabular} 
    \caption{Results for polynomials of degree 4. A "$-$" on the SOS method indicates that it ran out of memory and a $\dag$ indicates that a lower bound was returned without a certificate of optimality. \label{tab:results_degree4}}
\end{table}   

From Table \ref{tab:results_degree4}, we observe that our method achieves the best performance in all cases for $n=20$ and $n=15$. Moreover, for $n=10$ on $\mathcal{X}_1$ SOS is faster, while on both $\mathcal{X}_2'$ and $\mathcal{X}_3$ our method achieves the best performance. Finally, we note that on $\mathcal{X}_2'$ SOS derives a lower bound that is not equal to the optimum.

\section{Conclusions}  
In this paper, we developed a new approach for polynomial optimization, by deriving SLC decompositions for polynomials. Further, out of the infinitely many possible decompositions, we derived the one that results in the tightest lower bound for the original problem. The resulting problem is a SDP, in which the dimension of the largest LMI is reduced compared to the SOS method. As a result, our approach often outperforms state-of-the-art methods for polynomial optimization, such as BARON and SOS. Moreover, in the numerical experiments we showed that with our method, we can efficiently solve polynomial optimization problems with 40 variables and degree 3, as well as 20 variables and degree 4.

\bibliography{Masterbib}
\bibliographystyle{plainnat}

\begin{appendices} 
\section{Gershgorin formulation for degree 3} 
We use extra variables $\boldsymbol{T}^i,\boldsymbol{S}^i, \boldsymbol{T}'$ to linearize $\left \lvert \boldsymbol{P}^i \right \rvert, \left \lvert \boldsymbol{Q}^i \right \rvert, \left \lvert \boldsymbol{P}' \right \rvert$, respectively. Let $\mathcal{Z}_2$ denote the set of constraints that the coefficients $\boldsymbol{Z}$ need to satisfy in this case, which is defined as follows
\begin{align*}
    \mathcal{Z}_2 = \left\{ \begin{array}{lr} 
     \left( \boldsymbol{P}^i, \boldsymbol{Q}^i, \boldsymbol{P}', \boldsymbol{r}^i, \boldsymbol{f}^i, \boldsymbol{r}', w^i, g^i, w' \right) \in \mathcal{Z}_1, \\ 
      \sum_{i=1}^n \left \langle \Hat{\boldsymbol{A}}_j^i, \boldsymbol{P}^i \right \rangle + \left \langle \Hat{\boldsymbol{B}}_j^i, \boldsymbol{Q}^i \right \rangle + \left \langle \boldsymbol{H}_j^i, \boldsymbol{T}^i \right \rangle + \left \langle \boldsymbol{G}_j^i, \boldsymbol{S}^i \right \rangle \\ 
      \quad \quad \quad + \left( \Hat{\boldsymbol{c}}_j^i \right)^T \boldsymbol{r}^i + \left( \Hat{\boldsymbol{d}}_j^i \right)^T \boldsymbol{f}^i + \Hat{\mu}_j^i w^i + \Hat{\nu}_j^i g^i \\ 
      \quad + \left \langle \boldsymbol{\Psi}_j, \boldsymbol{T}' \right \rangle + \left \langle \Hat{\boldsymbol{\Xi}}_j, \boldsymbol{P}' \right \rangle + \Hat{\boldsymbol{\omega}}^T \boldsymbol{r}' + \hat{\gamma}w'  \le s_j^2, \: j \in [m_2].
    \end{array} \right\},
\end{align*} 
We have the following result for $\mathcal{Z}_2$. 
\begin{theorem}
\label{thm:gersh_degree3}
    Assuming the set $\mathcal{Z}_2$, then for a given $(\boldsymbol{x},\boldsymbol{U})$, the tightest lower bound for $g_3(\boldsymbol{x},\boldsymbol{U},\boldsymbol{Z})$ corresponds to the solution of the following problem

\begin{subequations}
			\begin{align}
				\displaystyle \min_{\boldsymbol{x}, \boldsymbol{U}, \boldsymbol{Y}, \boldsymbol{R}, \boldsymbol{\lambda}, \boldsymbol{\theta}} &\quad \sum_{j=1}^{m_2} \theta_j s_j^2 - \sum_{j=1}^{m} \lambda_j s_j \nonumber \\
				{\rm s.t.} &\quad \boldsymbol{Y}^i + \sum_{j=1}^{m} \lambda_j \boldsymbol{A}_j^i - \sum_{j=1}^{m_2} \theta_j \Hat{\boldsymbol{A}}_j^i = \boldsymbol{0}, \quad i \in [n], \label{eq:gersh_degree3_a} \\ 
                &\quad \boldsymbol{R}^i + \sum_{j=1}^{m} \lambda_j \boldsymbol{B}_j^i - \sum_{j=1}^{m_2} \theta_j \Hat{\boldsymbol{B}}_j^i = \boldsymbol{0}, \quad i \in [n],,  \label{eq:eq:gersh_degree3_b} \\ 
                &\quad \boldsymbol{Y}^{n+1} + \sum_{j=1}^{m} \lambda_j \boldsymbol{\Xi}_j - \sum_{j=1}^{m_2} \theta_j \Hat{\boldsymbol{\Xi}}_j = \boldsymbol{0}, \label{eq:eq:gersh_degree3_c} \\ 
                 &\quad \boldsymbol{U}_i + \sum_{j=1}^{m} \lambda_j \boldsymbol{c}_j^i - \sum_{j=1}^{m_2} \theta_j \Hat{\boldsymbol{c}}_j^i = \boldsymbol{0}, \quad i \in [n], \label{eq:eq:gersh_degree3_d} \\ 
                &\quad \boldsymbol{x}-\boldsymbol{U}_i + \sum_{j=1}^{m} \lambda_j \boldsymbol{d}_j^i - \sum_{j=1}^{m_2} \theta_j \Hat{\boldsymbol{d}}_j^i = \boldsymbol{0}, \quad i \in [n], \label{eq:eq:gersh_degree3_e} \\  
                &\quad \boldsymbol{x} + \sum_{j=1}^{m} \lambda_j \boldsymbol{\omega}_j - \sum_{j=1}^{m_2} \theta_j \Hat{\boldsymbol{\omega}}_j = \boldsymbol{0}, \label{eq:eq:gersh_degree3_f} \\ 
                &\quad x_i + \sum_{j=1}^{m} \lambda_j \mu_j^i - \sum_{j=1}^{m_2} \theta_j \Hat{\mu}_j^i = 0, \quad i \in [n], \label{eq:eq:gersh_degree3_g} \\ 
                &\quad 1-x_i + \sum_{j=1}^{m} \lambda_j \nu_j^i - \sum_{j=1}^{m_2} \theta_j \Hat{\nu}_j^i = 0, \quad i \in [n], \label{eq:gersh_degree3_h} \\ 
                &\quad 1 + \sum_{j=1}^{m} \lambda_j \gamma_j - \sum_{j=1}^{m_2} \theta_j \Hat{\gamma}_j = 0, \label{eq:gersh_degree3_i} \\ 
                &\quad \sum_{j=1}^{m_2} \theta_j \boldsymbol{H}_j^i = \boldsymbol{0}, \quad i \in [n], \label{eq:gersh_degree3_j} \\ 
                &\quad \sum_{j=1}^{m_2} \theta_j \boldsymbol{G}_j^i = \boldsymbol{0}, \quad i \in [n], \label{eq:gersh_degree3_k} \\ 
                &\quad \sum_{j=1}^{m_2} \theta_j \boldsymbol{\Psi}_j = \boldsymbol{0}, \label{eq:gersh_degree3_l} \\
                &\quad \begin{pmatrix}
                        \boldsymbol{Y}^i & \boldsymbol{U}_i \\
                        \boldsymbol{U}_i^T & x_i
                   \end{pmatrix} \succeq \boldsymbol{0}, \quad i \in [n], \label{eq:gersh_degree3_m} \\ 
               &\quad \begin{pmatrix}
                        \boldsymbol{R}^i & \boldsymbol{x}-\boldsymbol{U}_i \\
                        \left(\boldsymbol{x}-\boldsymbol{U}_i\right)^T & 1-x_i
                   \end{pmatrix} \succeq \boldsymbol{0}, \quad i \in [n], \label{eq:gersh_degree3_n} \\ 
             &\quad \begin{pmatrix}
                        \boldsymbol{Y}^{n+1} & \boldsymbol{x} \\
                        \boldsymbol{x}^T & 1
                   \end{pmatrix} \succeq \boldsymbol{0}, \label{eq:gersh_degree3_o} \\
              &\quad \left( \boldsymbol{x}, \boldsymbol{U} \right) \in \overline{\mathcal{X}}, \boldsymbol{\theta} \ge \boldsymbol{0}. \label{eq:gersh_degree3_p}
			\end{align}
\end{subequations}
\end{theorem} 

\begin{proof}{Proof. }
From Theorem 13.1 in \cite{bertsimas22}, and leveraging duality, we have that the following problem  
\begin{align*}
    \max_{\boldsymbol{Z} \in \mathcal{Z}_2} g_3(\boldsymbol{x},\boldsymbol{U},\boldsymbol{Z}).
\end{align*}  
is equivalent to the objective of the Theorem.
Further, the extra variables $\boldsymbol{Y}^i, \boldsymbol{R}^i$ need to satisfy the additional LMIs 
\begin{align*}
    \begin{pmatrix}
                \boldsymbol{Y}^i & \boldsymbol{U}_i \\
                \boldsymbol{U}_i^T & x_i
           \end{pmatrix} \succeq \boldsymbol{0}, \quad  \begin{pmatrix}
                \boldsymbol{R}^i & \boldsymbol{x}-\boldsymbol{U}_i \\
                \left(\boldsymbol{x}-\boldsymbol{U}_i\right)^T & 1-x_i
           \end{pmatrix} \succeq \boldsymbol{0}.
\end{align*} \hfill $\square$
\end{proof}

\section{Existence of first type of SLC decomposition for arbitrary degree}
In this section, we prove the existence of an SLC decomposition for an arbitrary polynomial of degree $d$, denoted as $p_d$, where the linear terms are defined by $x_i$ and $1-x_i$, and the convex terms are polynomials of degree $d-1$. 
\begin{theorem} 
\label{thm:arb_degree_first}
Every polynomial of degree $d$, denoted as $p_d(\boldsymbol{x})$, can be written as 
\begin{align*}
    p_d(\boldsymbol{x}) = \sum_{i=1}^n x_i p_{d-1}^i(\boldsymbol{x}) + \sum_{i=1}^n \left(1 - x_i \right) q_{d-1}^i(\boldsymbol{x}),
\end{align*}  
where $p_{d-1}^i(\boldsymbol{x}), q_{d-1}^i(\boldsymbol{x})$ are convex polynomials of degree $d-1$. 
\end{theorem} 

\begin{proof}{Proof. } 
We can easily construct polynomials $p_{d-1}^i, q_{d-1}^i$ of degree $d-1$, such that 
\begin{align*}
    p_d(\boldsymbol{x}) = \sum_{i=1}^n x_i p_{d-1}^i(\boldsymbol{x}) + \sum_{i=1}^n (1-x_i) q_{d-1}^i(\boldsymbol{x}),
\end{align*} 
for any $\boldsymbol{x} \in [0,1]^n$. Let $H(\boldsymbol{x})$ denote the Hessian matrix of $p_{d-1}^i(\boldsymbol{x})$. We add the value $\alpha$ to the diagonal in order to make the Hessian matrix diagonally dominant, that is,  
\begin{align*}
    H_{ii}(\boldsymbol{x}) + \alpha \ge \sum_{j \neq i} |H_{ij}(\boldsymbol{x})|, \ \forall i, \: \forall \boldsymbol{x}.
\end{align*}
Observe that the $(i,j)$-th entry of $H(\boldsymbol{x})$ is a polynomial of degree $n-3$, that is,
\begin{align*}
    H_{ij}(\boldsymbol{x}) = \sum_{k=1}^{n-3} \sum_{i_1,..,i_k} c_{i_1,..,i_k}^{ij, k} x_{i_1}..x_{i_k},
\end{align*}  
where $c_{i_1,..,i_k}^{ij, k}$ denote coefficients of product terms of degree $k$, that involve $x_i x_j$. It follows that 
\begin{align*}
    |H_{ij}| \le \sum_{k=1}^{n-3} \sum_{i_1,..,i_k} \left \lvert c_{i_1,..,i_k}^{ij, k} x_{i_1}..x_{i_k} \right \rvert \le \sum_{k=1}^{n-3} \sum_{i_1,..,i_k} \left \lvert c_{i_1,..,i_k}^{ij, k} \right \rvert, 
\end{align*}
since $\left \lvert x_{i_1}..x_{i_k} \right \rvert \le 1$, for all $k$. Similarly, we obtain that 
\begin{align*}
    H_{ii} = \sum_{k=1}^{n-3} \sum_{i_1,..,i_k} c_{i_1,..,i_k}^{ii, k} x_{i_1}..x_{i_k} \ge -\sum_{k=1}^{n-3} \sum_{i_1,..,i_k} \left \lvert c_{i_1,..,i_k}^{ii, k} \right \rvert,
\end{align*}  
where $c_{i_1,..,i_k}^{ii, k}$ denote coefficients of product terms of degree $k$ that involve $x_i^2$. Therefore, to ensure that $H(\boldsymbol{x})$ is diagonally dominant, it suffices to have 
\begin{align*}
    \alpha \ge \sum_{j \neq i} \sum_{k=1}^{n-3} \sum_{i_1,..,i_k} \left \lvert c_{i_1,..,i_k}^{ij, k} \right \rvert + \sum_{k=1}^{n-3} \sum_{i_1,..,i_k} \left \lvert c_{i_1,..,i_k}^{ii, k} \right \rvert.
\end{align*}
Therefore, by adding a term $\alpha_p^i$ in the diagonal, such that 
\begin{align*}
    \alpha_p^i \ge \max_l \left\{ \sum_{j \neq l} \sum_{k=1}^{n-3} \sum_{i_1,..,i_k} \left \lvert c_{i_1,..,i_k}^{lj, k} \right \rvert + \sum_{k=1}^{n-3} \sum_{i_1,..,i_k} \left \lvert c_{i_1,..,i_k}^{ll, k} \right \rvert \right\},
\end{align*}
we can make the Hessian matrix of $p_{d-1}^i(\boldsymbol{x})$ diagonally dominant. Similarly, we define $\alpha_q^i$, and then we take $\alpha = \max_i \max\{\alpha_p^i, \alpha_q^i\}$. Consider the polynomials
\begin{align*}
    & \overline{p}_{d-1}^i(\boldsymbol{x}) = p_{d-1}^i(\boldsymbol{x}) + \alpha \|\boldsymbol{x}\|^2 - \alpha n x_i  \\
    & \overline{q}_{d-1}^i(\boldsymbol{x}) = q_{d-1}^i(\boldsymbol{x}) + \alpha \|\boldsymbol{x}\|^2,
\end{align*}     
that are both convex. Further, they satisfy
\begin{align*}
    \sum_{i=1}^n x_i\overline{p}_{d-1}^i(\boldsymbol{x}) + \sum_{i=1}^n (1-x_i) \overline{q}_{d-1}^i(\boldsymbol{x}) &= \sum_{i=1}^n x_ip_{d-1}^i(\boldsymbol{x}) + \sum_{i=1}^n (1-x_i) q_{d-1}^i(\boldsymbol{x}) \\ 
    &= p_d(\boldsymbol{x}).
\end{align*} \hfill $\square$
\end{proof} 

\end{appendices}

\end{document}